\documentclass[11pt]{article}

\usepackage{amsthm}
\usepackage{amssymb}
\usepackage{authblk}

\usepackage{amsfonts,amsmath}
\usepackage{hyperref}
\usepackage{graphicx}
\usepackage{caption}
\usepackage{subcaption}

\makeatletter
\g@addto@macro\normalsize{%
}
\makeatother

\newcommand{\beq}[1]{\begin{equation} \label{#1}}
\newcommand{\eeq}{\end{equation}}
\newcommand{\bed}{\begin{displaymath}}
\newcommand{\eed}{\end{displaymath}}
\newcommand{\bea}{\bed\begin{array}{rl}}
\newcommand{\eea}{\end{array}\eed}

\newcommand{\barray}{\begin{array}{ll}}
\newcommand{\earray}{\end{array}}
\def\({\left(}
\def\){\right)}

\newtheorem{theorem}{Theorem}
\newtheorem{lemma}{Lemma}
\newtheorem{assumption}{Assumption}
\newtheorem{prop}{Proposition}
\newtheorem{defi}{Definition}
\newtheorem{remark}{Remark}
\newtheorem{example}{Example}

\title{Convergence Rate of LQG Mean Field Games with Common Noise}

\author[]{Jiamin Jian }
\author[]{Qingshuo Song \thanks{qsong@wpi.edu}}
\author[]{Jiaxuan Ye }
\affil[]{Department of Mathematical Sciences, Worcester Polytechnic Institute}

\date{}

\begin{document}

\maketitle


\begin{abstract}
This paper focuses on exploring the convergence properties of a generic player's trajectory and empirical measures in an $N$-player Linear-Quadratic-Gaussian Nash game, where Brownian motion serves as the common noise. The study establishes three distinct convergence rates concerning the representative player and empirical measure.
To investigate the convergence, the methodology relies on a specific decomposition of the equilibrium path in the $N$-player game and utilizes the associated Mean Field Game framework.
\end{abstract}

\section{Introduction}

Mean Field Game (MFG) theory was introduced by Lasry and Lions in their seminal paper (\cite{LJ07}), and by Huang, Caines, and Malhame (\cite{HMC06, HCM07_1, HCM07_2, HCM07_3}). It aims to provide a framework for studying the asymptotic behavior of $N$-player differential games being invariant under the reshuffling of the players' indices. For a comprehensive overview of recent advancements and relevant applications of MFG theory, it is recommended to refer to the two-volume book by Carmona and Delarue (\cite{CD18I, CD18II}) published in 2018 and the references provided therein.

Mean Field Games (MFG) have become widely accepted as an approximation for $N$-player games, particularly when the number of players, $N$, is large enough. A fundamental question that arises in this context concerns the convergence rate of this approximation.
Convergence can be analyzed from different perspectives, such as convergence in value, the trajectory followed by the representative player, or the behavior of the mean field term. Each of these perspectives offers valuable insights into the behavior and characteristics of the MFG approximation. Furthermore, they raise a variety of intriguing questions within this context.

To be more concrete, we examine the behavior of the triangular array 
$\hat X_t^{(N)} = (\hat X_{it}^{(N)}: 1\le i \le N)$ as $N\to \infty$, 
where $\hat X_{it}^{(N)}$ represents the equilibrium state of the $i$-th player at time $t$ in the $N$-player game, defined within the probability space $\left(\Omega^{(N)}, \mathcal F^{(N)}, \mathbb F^{(N)}, \mathbb P^{(N)} \right)$. Additionally, we denote $\hat X_t$ as the equilibrium path at time $t$ derived from the associated MFG, defined in the probability space $(\Omega, \mathcal F, \mathbb F, \mathbb P)$.

Considering the identical but not independent distribution $\mathcal L(\hat X_{it}^{(N)})$, the first question pertains to the convergence of $\hat X_{1t}^{(N)}$, which represents the generic path. It can be framed as follows:

\begin{itemize}
\item [(Q1)]
The $\mathbb W_p$-convergence rate of the representative equilibrium path, 
$$\mathbb W_p \left(\mathcal L \left(\hat X_{1t}^{(N)} \right), \mathcal L \left(\hat X_t \right) \right) 
	= O \left(N^{-?} \right).$$
\end{itemize}
Here, $\mathbb W_p$ denotes the $p$-Wasserstein metric.

The existing literature extensively explores the convergence rate in this context. For (Q1), Theorem 2.4.9 of the monograph \cite{CDLL19} establishes a convergence rate of $O(N^{-1/2})$ using the $\mathbb W_1$ metric. More recently, \cite{JT23} addresses (Q1) by introducing displacement monotonicity and controlled common noise, and Theorem 2.23 applies the maximum principle of forward-backward propagation of chaos to achieve the same convergence rate. Within the LQG framework, \cite{JLSY22} also provides a convergence rate of $1/2$ for the representative player.

The second question pertains to the convergence of the mean-field term, which is equivalent to the convergence of the empirical measure $\rho (\hat X_t^{(N)}) = \frac{1}{N} \sum_{i=1}^{N} \delta_{\hat{X}_{it}^{(N)}}$ of $N$ players. Given the Brownian motion, denoted as $\tilde W_t$, to be the common noise, the problem lies in determining the rate of convergence of the empirical measures to the MFG equilibrium measure
$$ \hat m_t = \mathcal L \left( \left. \hat{X}_t \right \vert {\mathcal F}_t^{\tilde W} \right), \quad \forall t \in (0, T].$$
Thus, the second question can be stated as follows:
\begin{itemize}
\item [(Q2)] 
The $\mathbb W_p$-convergence rate of empirical measures in $L^p$ sense,
$$
\left(\mathbb E \left[ 
\mathbb W_p^p 
\left(\rho \left(\hat X_t^{(N)} \right), \mathcal L \left(\left. \hat X_t \right \vert {\mathcal F}_t^{\tilde W} \right) \right) 
\right] \right)^{\frac{1}{p}} = O \left(N^{-?} \right).
$$
\end{itemize}

 As for (Q2), Theorem 3.1 of \cite{DLR20} provides an answer, stating that the empirical measures 
 exhibit a convergence rate of $O(N^{-1/(2p)})$ in the $\mathbb W_p$ distance for $p \in [1,2]$. In \cite{DLR20}, they also explore a related question that is both similar and more intriguing, which concerns the uniform $\mathbb{W}_p$-convergence rate:
\begin{itemize}
\item [(Q3)] 
The $t$-uniform $\mathbb W_p$-convergence rate of empirical measures in $L^p$ sense,
$$
\left(\mathbb E \left[ 
\sup_{t\in [0, T]}
\mathbb W_p^p 
\left(\rho \left(\hat X_t^{(N)} \right), \mathcal L \left(\left. \hat X_t \right \vert {\mathcal F}_t^{\tilde W} \right) \right) 
\right] \right)^{\frac{1}{p}} = O \left(N^{-?} \right).
$$
\end{itemize}
The answer provided by Theorem 3.1 in \cite{DLR20} reveals that the uniform convergence rate, as formulated in (Q3), is considerably slower compared to the convergence rate mentioned in (Q2). Specifically, the convergence rate for (Q3) is $O\left(N^{-1/(d+8)} \right)$ when $p=2$, where $d$ represents the dimension of the state space.

In our paper, we specifically focus on a class of one-dimensional Linear-Quadratic-Gaussian (LQG) Mean Field Nash Games with Brownian motion as the common noise. It is important to note that the assumptions made in the aforementioned papers except \cite{JLSY22} only account for linear growth in the state and control elements for the running cost, thus excluding the consideration of LQG.
It is also noted that differences between \cite{JLSY22} and the current paper lie in various aspects: (1) The problem setting in our paper considers Brownian motion as the common noise, whereas \cite{JLSY22} employs a Markov chain. This discrepancy leads to significant differences in the subsequent analysis;  (2) The work in \cite{JLSY22} does not address the questions posed in (Q2) and (Q3).

Our main contribution is the establishment of the convergence rate of all three questions in the above in LQG framework.
Firstly, the paper establishes that the convergence rate of the $p$-Wasserstein metric for the distribution of the representative player is $O(N^{-1/2})$ for $p\in [1,2]$.
Secondly, it demonstrates that the convergence rate of the $p$-Wasserstein metric for the empirical measure in the $L^p$ sense is $O(N^{-1/(2p)})$ for $p\in [1,2]$.
Lastly, the paper shows that the convergence rate of the uniform $p$-Wasserstein metric for the empirical measure in the $L^p$ sense is $O(N^{-1/(2p)})$ for $p\in(1, 2]$, and $O(N^{-1/2} \ln (N))$ for $p = 1$.

It is worth noting that the convergence rates obtained for (Q1) and (Q2) in the LQG framework align with the results found in existing literature, albeit under different conditions. Additionally, it is revealed that the uniform convergence rate of (Q3) may be slower than that of (Q2), which is consistent with the observations made by \cite{DLR20} from a similar perspective.
Interestingly, when considering the specific case where $p=2$ and $d=1$, the uniform convergence rate of (Q3) is established as $O(N^{-1/9})$ according to \cite{DLR20}, while it is determined to be $O(N^{-1/4})$ within our framework that incorporates the LQG structure.

Regarding (Q2), if the states $(\hat X_{it}^{(N)}: 1\le i \le N)$ were independent, the convergence rate could be determined as $1/(2p)$ based on Theorem 1 of \cite{FG15} and Theorem 5.8 of \cite{CD18I}, which provide convergence rates for empirical measures of independent and identically distributed sequences. However, in the mean-field game, the states $\hat X_{it}^{(N)}$ are not independent of each other, despite having identical distributions.  
The correlation is introduced mainly by two factors: One is  the system coupling arising from the mean-field term and the other is the common noise.
Consequently, determining the convergence rate requires understanding the contributions of these two factors to the correlation among players. 

In our proof, we rely on a specific decomposition (refer to Lemma \ref{l:triangle} and the proof of the main theorem) of the underlying states. This decomposition reveals that the states can be expressed as a sum of a weakly correlated triangular array and a common noise. By analyzing the behavior of these components, we can address the correlation and establish the convergence rate.

Additionally, it is worth mentioning that a similar technique of dimension reduction in $N$-player LQG games have been previously utilized in \cite{HY21} and related papers to establish decentralized Nash equilibria and the convergence rate in terms of value functions.

The remainder of the paper is organized as follows: Section \ref{s:section2} outlines the problem setup and presents the main result. The proof of the main result, which relies on two propositions, is provided in Section \ref{s:proof}. 
We establish the proof for these two propositions in Section \ref{s:section3} and Section \ref{s:section4}. Some lemmas are given in the Appendix.


\section{Problem setup and main results}
\label{s:section2}

\subsection{The formulation of equilibrium in Mean Field Game}
\label{s:mfgsetting}

In this section, we present the formulation of the Mean Field Game in the sample space $\Omega$.

Let $T > 0$ be a given time horizon. We assume that $W = \{W_t\}_{t \geq 0}$ is a standard Brownian motion constructed on the probability space $(\bar{\Omega}, \bar{\mathcal{F}} = \bar{\mathcal{F}}_T, \bar{\mathbb{P}}, \bar{\mathbb{F}} = \{\bar{\mathcal{F}}_t\}_{t \geq 0})$. Similarly, the process $\tilde{W} = \{\tilde{W}_t\}_{t \geq 0}$ is a standard Brownian motion constructed on the probability space $(\tilde{\Omega}, \tilde{\mathcal{F}} = \tilde{\mathcal{F}}_T, \tilde{\mathbb{P}}, \tilde{\mathbb{F}} = \{\tilde{\mathcal{F}}_t\}_{t \geq 0})$. We define the product structure as follows:
$$\Omega = \bar{\Omega} \times \tilde{\Omega}, \quad \mathcal F, \quad \mathbb F = \{\mathcal F_t\}_{t \geq 0}, \quad \mathbb P,$$
where $(\mathcal F, \mathbb P)$ is the completion of $( \bar{\mathcal F} \otimes \tilde{\mathcal F}, \bar{\mathbb P} \otimes \tilde{\mathbb P} )$ and $\mathbb F$ is the complete and right continuous augmentation of $\{\bar{\mathcal F}_t \otimes \tilde{\mathcal F}_t \}_{t \geq 0}$.

Note that, $W$ and $\tilde W$ are two Brownian motions from separate sample spaces $\bar \Omega$ and $\tilde \Omega$, they are independent of each other in their product space $\Omega$.
 In our manuscript, $W$ is called individual or idiosyncratic noise, and $\tilde{W}$ is called common noise, 
 see their different roles in the problem formulation later defined via fixed point condition \eqref{eq:mhat}. 
To proceed, we denote by $L^p := L^p(\Omega, \mathbb P)$ the set of random variables $X$ on $(\Omega, \mathcal F, \mathbb P)$ with finite $p$-th moment with norm $\|X\|_p = (\mathbb E \left[|X|^p \right])^{1/p}$ 
and by $L_{\mathbb F}^p:=L_{\mathbb F}^p(\Omega \times [0, T])$ the space of all $\mathbb R$ valued $\mathbb F$-progressively measurable random processes $\alpha$ such that
$$\mathbb E \left[ \int_0^T |\alpha_t|^p dt \right] <\infty.$$

Let $\mathcal P_p(\mathbb R)$ denote the Wasserstein space of probability measures $\mu$ on $\mathbb R$ satisfying 
$\int_{\mathbb R} x^p d \mu(x) < \infty$
endowed with $p$-Wasserstein metric  $\mathbb W_p(\cdot, \cdot)$ defined by
$$\mathbb W_p(\mu, \nu) = \inf_{\pi \in \Pi(\mu, \nu)} \left( \int_{\mathbb R \times \mathbb R} |x - y|^p d\pi(x, y) \right)^{\frac{1}{p}},$$
where $\Pi(\mu, \nu)$ is the collection of all probability measures on $\mathbb R \times \mathbb R$ with its marginals agreeing with $\mu$ and $\nu$.

Let $X_0\in L^2$ be a random variable that is independent with $W$ and $\tilde{W}$. For any control $\alpha \in L^2_{\mathbb F}$, consider the state $X = \{X_t\}_{t \geq 0}$ of the generic player is governed by a stochastic differential equation (SDE)
\begin{equation}
\label{eq:X}
	d X_t = \alpha_t dt + d W_t + d \tilde{W}_t
\end{equation}
with the initial value $X_0$, where the underlying process $X: [0, T] \times \Omega \mapsto \mathbb R$. Given a random measure flow $m: (0, T]\times \Omega \mapsto \mathcal P_2(\mathbb R)$, the generic player wants to minimize the expected accumulated cost on $[0, T]$: 
\begin{equation}
	\label{eq:total cost}
	\begin{array}{ll}
		J (x, \alpha)  = \displaystyle
		\mathbb E \left[ \left.  \int_0^T \left( \frac{1}{2} \alpha_s^2 + F(X_s, m_{s}) \right) \, ds \right \rvert X_0 = x \right]
	\end{array}
\end{equation} 
with some given cost function $F: \mathbb R\times \mathcal P_2(\mathbb R) \mapsto \mathbb R$. 
	
The objective of the control problem for the generic player is to find its optimal control 
$\hat \alpha \in \mathcal A := L^4_{\mathbb F}$  to minimize the total cost, i.e.,
\begin{equation}
	\label{V_m}
	V[m](x) = J[m](x, \hat \alpha) \le J[m](x, \alpha), \quad \forall \alpha \in \mathcal A.
\end{equation}
Associated to the optimal control  $\hat \alpha$, we denote the optimal path by  $\hat{X}=\{ \hat X_t\}_{t \geq 0}$.

Next, to introduce the MFG Nash equilibrium, it is useful to emphasize the dependence of the optimal path and optimal control of the generic player, as well as its associated value, on the underlying measure flow $m$. These quantities are denoted as $\hat{X}_t[m]$, $\hat{\alpha}_t[m]$, $J[m]$, and $V[m]$, respectively.

We now present the definitions of the equilibrium measure, equilibrium path, and equilibrium control. Please also refer to page 127 of \cite{CD18II} for a general setup with a common noise.
\begin{defi}
	\label{d:ne}
	Given an initial distribution  $\mathcal L(X_0) = m_0 \in \mathcal P_2(\mathbb R)$, 
	a random measure flow $\hat m = \hat m(m_0)$ 
	is said to be an MFG equilibrium measure if it satisfies the fixed point  condition
	\begin{equation}
		\label{eq:mhat}
		\hat m_t = \mathcal L \left( \left.\hat X_t[\hat m] \right \rvert \tilde{\mathcal F}_t \right), \ \forall 0 < t \le T, \ \hbox{ almost surely in } \mathbb P.
	\end{equation}
	The path $\hat X$ and the control $\hat \alpha$ associated with $\hat m$ are called the MFG equilibrium path and equilibrium control, respectively. 
\end{defi}

\begin{figure}[h]
    \centering
    \includegraphics[width=.40\textwidth]{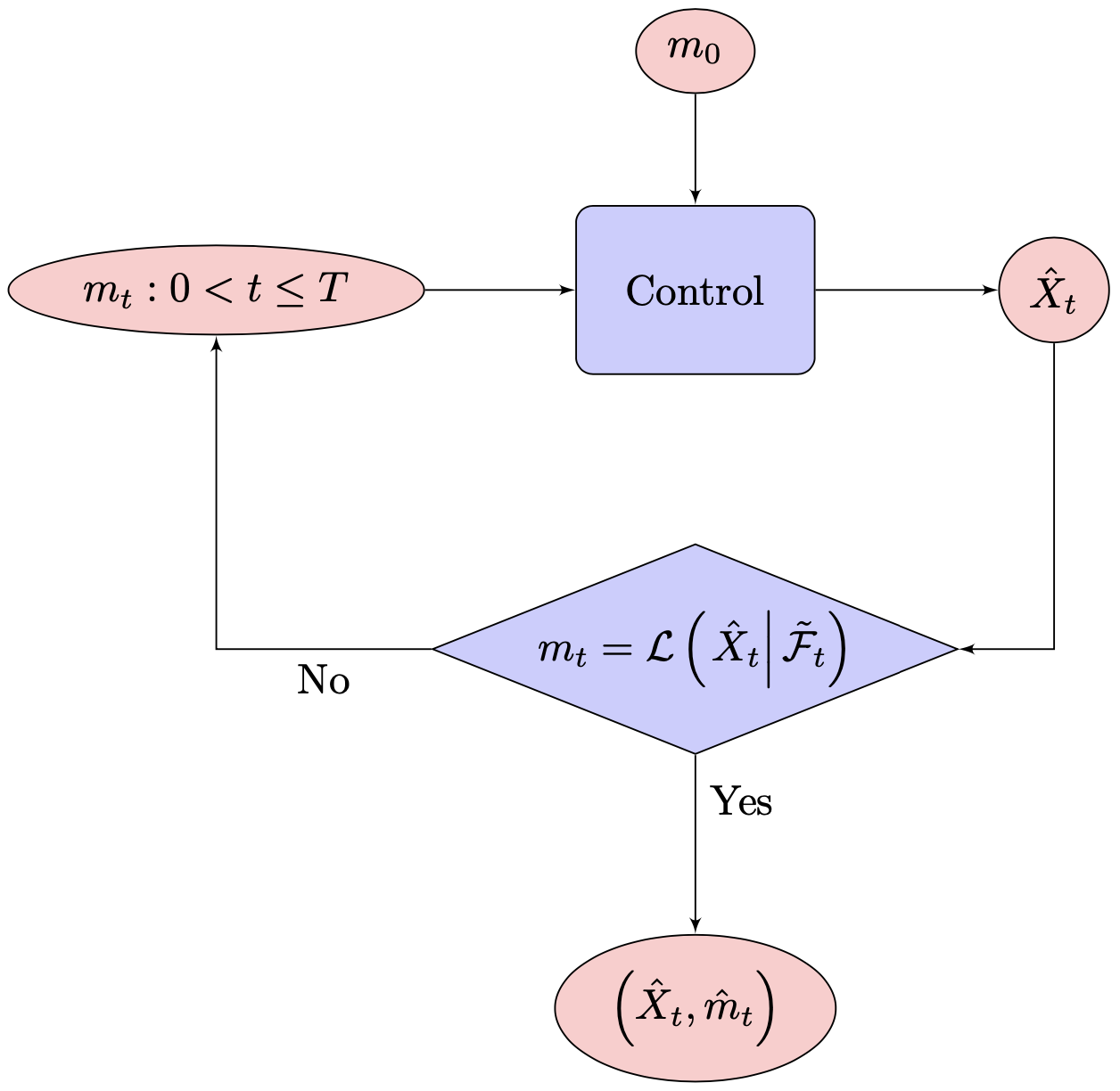}
    \caption{The MFG diagram.}
	\label{fig:MFG1}
\end{figure}

The flowchart of the MFG diagram is given in Figure \ref{fig:MFG1}. It is noted from the optimality condition \eqref{V_m} and the fixed point condition \eqref{eq:mhat} that 
$$J[\hat m](x, \hat \alpha) \le J[\hat m](x, \alpha), \quad \forall \alpha$$
holds for the equilibrium measure $\hat m$ and its associated equilibrium control $\hat \alpha$, 
while it is not
$$J[\hat m](x, \hat \alpha) \le J[m](x, \alpha), \quad \forall \alpha, m.$$
Otherwise, this problem turns into a McKean-Vlasov control problem, which is essentially different from the current Mean Field Games setup. Readers refer to \cite{CDL13, CD15} to see the analysis of this different model as well as some discussion of the differences between these two problems.

\subsection{The formulation of Nash equilibrium in $N$-player game}
\label{s:n-player}
In this subsection, we set up $N$-player game and define the Nash equilibrium of $N$-player game in the sample space $\Omega^{(N)}$.
Firstly, let $W^{(N)} = (W^{(N)}_{i}: i = 1, 2, \dots, N )$ be an $N$-dimensional standard Brownian motion constructed on the space $(\bar{\Omega}^{(N)}, \bar{\mathcal{F}}^{(N)}, \bar{\mathbb{P}}^{(N)}, \bar{\mathbb F}^{(N)} = \{ \bar{\mathcal{F}}^{(N)}_t \}_{t \geq 0})$ and 
$\tilde{W} = \{ \tilde{W}_t \}_{t \geq 0}$ be the common noise in MFG defined in 
Section \ref{s:mfgsetting} on $(\tilde{\Omega}, \tilde{\mathcal{F}}, \tilde{\mathbb{P}})$.
The probability space for the $N$-player game is 
$\left(\Omega^{(N)}, \mathcal F^{(N)}, \mathbb F^{(N)}, \mathbb P^{(N)} \right)$, which is constructed 
via the product structure with
$$\Omega^{(N)} = \bar{\Omega}^{(N)} \times \tilde{\Omega}, \quad \mathcal F^{(N)}, \quad \mathbb F^{(N)} = \left\{\mathcal F^{(N)}_t\right\}_{t \geq 0}, \quad \mathbb P^{(N)}.$$
where $(\mathcal F^{(N)}, \mathbb P^{(N)})$ is the completion of $( \bar{\mathcal F}^{(N)} \otimes 
\tilde{\mathcal F}, \bar{\mathbb P}^{(N)} \otimes \tilde{\mathbb P} )$ and $\mathbb F^{(N)}$ is the complete and right continuous augmentation of $\{\bar{\mathcal F}_t^{(N)} \otimes \tilde{\mathcal F}_t \}_{t \geq 0}$.

Consider a stochastic dynamic game with $N$ players, where each player $i \in \{1, 2, \dots, N\}$ controls a state process $X_{i}^{(N)} = \{X_{it}^{(N)} \}_{t \geq 0}$ in $\mathbb R$ given by
\begin{equation}
\label{eq:Xi}
	d X_{it}^{(N)} = \alpha_{it}^{(N)} d t + d W_{it}^{(N)} + d \tilde{W}_t, \quad X_{i0}^{(N)} = x^{(N)}_i
\end{equation}
with a control $\alpha_{i}^{(N)}$ in an admissible set $\mathcal A^{(N)} := L^4_{\mathbb F^{(N)}}$ and random initial state $x^{(N)}_i$.
	
Given the strategies $\alpha_{-i}^{(N)} = (\alpha_{1}^{(N)}, \dots, \alpha_{i-1}^{(N)}, \alpha_{i+1}^{(N)}, \dots,  \alpha_{N}^{(N)} )$ from other players, the objective of player $i$ is to select a control $\alpha_{i}^{(N)} \in \mathcal A^{(N)}$ to minimize her expected total cost given by
\begin{equation}
\label{eq:cost_N}
\begin{aligned}
	J_i^{N} \left(x^{(N)}, \alpha_i^{(N)};  \alpha_{-i}^{(N)} \right) 
	&= \mathbb E\left[ \left.\int_0^T 
	\left(\frac{1}{2}\left(\alpha_{it}^{(N)} \right)^2 
	+ F \left(X_{it}^{(N)}, \rho \left(X_t^{(N)} \right) \right) \right) dt \right \rvert X_{0}^{(N)} = x^{(N)} \right],
\end{aligned}
\end{equation}
where $x^{(N)} = (x_1^{(N)}, x_2^{(N)}, \dots, x_N^{(N)} )$ is a $\mathbb R^N$-valued random vector in $\Omega^{(N)}$ to denote the initial state for $N$ players, and 
$$\rho \left(x^{(N)} \right) = \frac{1}{N} \sum_{i=1}^N \delta_{x_i^{(N)}}$$
is the empirical measure of the vector $x^{(N)}$ with Dirac measure $\delta$. We use the notation
$\alpha^{(N)} := (\alpha_{i}^{(N)}, \alpha_{-i}^{(N)} ) = (\alpha_1^{(N)}, \alpha_2^{(N)}, \ldots, \alpha_N^{(N)})$ to denote the control from $N$ players as a whole.
Next, we give the equilibrium value function and equilibrium path in the sense of the Nash game.
	
\begin{defi}
\label{d:neN}
\begin{enumerate}
	\item The value function of player $i$  for $i = 1, 2, \ldots, N$ of the Nash game is defined by $V^N = (V^N_i: i = 1, 2, \ldots, N )$ satisfying the equilibrium condition
	\begin{equation}
	\label{eq:value_i}
	V_i^{N}\left(x^{(N)} \right) := J_i^{N} \left(x^{(N)}, \hat \alpha_i^{(N)}; \hat \alpha_{-i}^{(N)} \right) \le J_i^{N} \left(x^{(N)}, \alpha_i^{(N)}; \hat \alpha_{-i}^{(N)} \right), \quad \forall \alpha_i^{(N)} \in \mathcal A^{(N)}.
	\end{equation}
			
	\item The equilibrium path of the $N$-player game is the $N$-dimensional random path $\hat X_t^{(N)} = (\hat X_{1t}^{(N)}, \hat X_{2t}^{(N)}, \ldots, \hat X_{Nt}^{(N)} )$ driven by \eqref{eq:Xi} associated to the control $\hat \alpha_t^{(N)}$ satisfying the equilibrium condition of \eqref{eq:value_i}.
\end{enumerate}
\end{defi}

	
\subsection{Main result}
\label{s:main}
We consider three convergence
questions on $N$-player game defined in $\Omega^{(N)}$:
The first one is the convergence of the representative path $\hat X_{it}^{(N)}$, 
the second one is the convergence  of the empirical measure $\rho (\hat X_{t}^{(N)} )$, while the last one is the $t$-uniform convergence  of the empirical measure $\rho (\hat X_{t}^{(N)})$. 
To be precise, we shall assume the following throughout the paper:
\begin{assumption}
\label{a:asm1}
\begin{itemize}
\item $\mathbb E [ |X_0|^q ]<\infty$ for some $q>4$.
\item
The initials $X_{i0}^{(N)}$ of the $N$-player game is  i.i.d. random variables in $\Omega^{(N)}$ 
with the same distribution as $\mathcal L(X_0)$ in the MFG.
\end{itemize}
\end{assumption}

Note that the equilibrium path $\hat X_t^{(N)} = (\hat X_{it}^{(N)}: i = 1, 2, \ldots, N)$ is a 
vector-valued stochastic process. 
Due to the Assumption \ref{a:asm1}, the game is invariant to index reshuffling of $N$ players and the elements in $(\hat X_{it}^{(N)}: i = 1, 2, \ldots, N)$
have identical distributions, but they are not independent of each other. 

So, the first question on the representative path is indeed about $\hat X_{1t}^{(N)}$ in $\Omega^{(N)}$ and we are interested in how fast it converges to $\hat X_{t}$ in $\Omega$ in distribution:
\begin{itemize}
\item [(Q1)]
The $\mathbb W_p$-convergence rate of the representative equilibrium path, 
$$\mathbb W_p \left(\mathcal L \left(\hat X_{1t}^{(N)} \right), \mathcal L \left(\hat X_t \right) \right) 
	= O \left(N^{-?} \right).$$
\end{itemize}
The second question is about the convergence of the empirical measure  $\rho (\hat X_t^{(N)} )$ 
of the $N$-player game defined by
$$\rho \left(\hat X_t^{(N)} \right) =\frac{1}{N} \sum_{i=1}^{N} \delta_{\hat{X}_{it}^{(N)}}.$$
We are interested in how fast this converges to the MFG equilibrium measure given by
$$ \hat m_t = \mathcal L \left( \left. \hat{X}_t \right \vert \tilde{\mathcal F}_t \right), \quad \forall t \in (0, T].$$
\begin{itemize}
\item [(Q2')] 
The $\mathbb W_p$-convergence rate of empirical measures,
$$
\mathbb W_p 
\left(\rho \left(\hat X_t^{(N)} \right), \mathcal L \left(\left. \hat X_t \right \vert \tilde{\mathcal F}_t \right) \right) 
	= O \left(N^{-?} \right).
$$
\end{itemize}

Note that the left-hand side of the above equality is a random quantity 
and one shall be more precise about what the Big $O$ notation means in this context. 
Indeed, by the definition of the empirical measure, $\rho (\hat X_t^{(N)} )$ is a random distribution measurable by 
$\sigma$-algebra generated by the random vector $\hat X_t^{(N)}$.
On the other hand, $\mathcal L (\hat X_t \rvert \tilde{\mathcal F}_t)$ is a random distribution measurable by the $\sigma$-algebra $\tilde{\mathcal F}_t$. Therefore, 
from the construction of the product probability space $\Omega^{(N)}$ in Section \ref{s:n-player}, 
both random distributions $\rho (\hat X_t^{(N)} )$ and $\mathcal L (\hat X_t \rvert \tilde{\mathcal F}_t )$ 
are measurable with respect to $\mathcal F^{(N)}_t = \bar{\mathcal F}_t^{(N)} \otimes \tilde{\mathcal F}_t$. 
Consequently, $\mathbb W_p (\rho (\hat X_t^{(N)}), \mathcal L (\hat X_t \rvert \tilde{\mathcal F}_t ))$ 
is a random variable in the probability space
$(\Omega^{(N)}, \mathcal F^{(N)}, \mathbb P^{(N)} )$ and we will focus on a version of (Q2') in the $L^p$ sense:
\begin{itemize}
\item [(Q2)] 
The $\mathbb W_p$-convergence rate of empirical measures in $L^p$ sense for each $t \in [0, T]$,
$$
\left ( \mathbb E \left [ 
\mathbb W_p^p 
\left(\rho \left(\hat X_t^{(N)} \right), \mathcal L \left(\left. \hat X_t \right \vert \tilde{\mathcal F}_t \right) \right) 
\right] \right)^{\frac{1}{p}} = O \left(N^{-?} \right).
$$
\end{itemize}
In addition, we also study the following related question:
\begin{itemize}
\item [(Q3)] 
The $t$-uniform $\mathbb W_p$-convergence rate of empirical measures in $L^p$ sense,
$$
\left ( \mathbb E \left [ 
\sup_{0\le t \le T}
\mathbb W_p^p 
\left(\rho \left(\hat X_t^{(N)} \right), \mathcal L \left(\left. \hat X_t \right \vert \tilde{\mathcal F}_t \right) \right) 
\right] \right)^{\frac{1}{p}} = O \left(N^{-?} \right).
$$
\end{itemize}

In this paper, we will study the above three questions (Q1), (Q2), and (Q3)
in the framework of LQG structure with Brownian motion as a common noise with the following function $F$ 
in the cost functional \eqref{eq:total cost}.
\begin{assumption}
\label{a:asm2}
Let the function $F: \mathbb R \times \mathcal P_2(\mathbb R) \mapsto \mathbb R$ be given in the form of 
\begin{equation}
\label{eq:running cost}
	F(x, m) = k \int_{\mathbb R} (x-z)^2 m(dz) = k(x^2 -2x[m]_1 + [m]_2)
\end{equation}
for some $k>0$, where $[m]_1, [m]_2$ are the first and second moment of the measure $m$. 
\end{assumption}

The main result of this paper is presented below. Let us recall that $q$ denotes the parameter defined in Assumption \ref{a:asm1}.\begin{theorem}
\label{t:main1}
Under Assumptions \ref{a:asm1}-\ref{a:asm2}, for any $p\in [1,2]$, we have
\begin{enumerate}
\item
The $\mathbb W_p$-convergence rate of the representative equilibrium path is $1/2$, i.e., 
$$\mathbb W_p \left(\mathcal L \left(\hat X_{1t}^{(N)} \right), \mathcal L \left(\hat X_t \right) \right) 
	= O \left(N^{-\frac 1 2} \right).$$
\item The $\mathbb W_p$-convergence rate of empirical measures in $L^p$ sense is 
$$
\mathbb E \left[
\mathbb W_p^p
\left(\rho \left(\hat X_t^{(N)} \right), \mathcal L \left(\left. \hat X_t \right \vert \tilde{\mathcal F}_t \right) \right) 
\right] 
= O \left(N^{-\frac{1}{2}} \right).
$$
\item  The uniform $\mathbb W_p$-convergence rate of empirical measures in $L^p$ sense is 
$$
\mathbb E \left[
\sup_{0\le t \le T} \mathbb W_p^p
\left(\rho \left(\hat X_t^{(N)} \right), \mathcal L \left(\left. \hat X_t \right \vert \tilde{\mathcal F}_t \right) \right) 
\right] 
= 
\begin{cases}
O\left(N^{-\frac{1}{2}} \ln (N) \right), & \hbox{ if } p = 1, \\
O \left(N^{-\frac{1}{2}} \right), & \hbox{ if }  1< p \leq 2.
\end{cases}
$$
\end{enumerate}
\end{theorem}

We would like to provide some additional remarks on our main result. Firstly, the cost function $F$ defined in \eqref{eq:cost_N} applies to the running cost for the $i$-th player in the $N$-player game, and it takes the form:
\begin{equation}
\label{eq:running cost_n player}
    F \left(X_{it}^{(N)}, \rho \left(X_t^{(N)} \right) \right)  =
	\frac{k}{N}\sum_{j = 1}^N \left(X_{it}^{(N)}-X_{jt}^{(N)} \right)^2.
\end{equation}
Interestingly, if $k<0$, although $F$ does satisfy the Lasry-Lions monotonicity (\cite{Car10}) as demonstrated in Appendix 6.1 of \cite{JLSY22}, there is no global solution for MFG due to the concavity in $x$. 
On the contrary, when $k>0$, 
$F$ satisfies the displacement monotonicity proposed in \cite{GMMZ22} as shown by the following derivation: 
$$\mathbb E \left[(F_x(X_1, \mathcal L(X_1))- F_x(X_2, \mathcal L(X_2))) (X_1 - X_2) \right] = 2 k \left(\mathbb E \left[(X_1 - X_2)^2 \right] - \left(\mathbb E [X_1 - X_2] \right)^2 \right) \ge 0.$$


\section{Proof of the main result with two propositions}\label{s:proof}
Our objective is to investigate the relations between $(\hat X_{1t}^{(N)}, \hat X_{2t}^{(N)}, \ldots, \hat X_{Nt}^{(N)})$ and $\hat X_t$ 
described in (Q1), (Q2), and (Q3). 
In this part, we will give the proof of Theorem \ref{t:main1} based on two propositions whose proof will be given later.
\begin{prop}
\label{p:mfg}
Under Assumptions \ref{a:asm1}-\ref{a:asm2}, the MFG equilibrium path $\hat X = \hat X [\hat m]$ is given by
\begin{equation}
\label{eq:Xhat03p}
	d \hat X_t = - 2 a(t) \left(\hat X_t - \hat \mu_t \right) dt + dW_t + d \tilde{W}_t, \quad \hat X_0 = X_0,
\end{equation}
where  $a$ is the solution of 	
\begin{equation}
\label{eq:a_odep}
a'(t) -2a^2(t) + k = 0, \quad a(T) = 0,
\end{equation}
and $\hat \mu$ is
$$ \hat \mu_t := \mathbb E \left[\left. \hat{X}_t \right\vert \tilde{\mathcal F}_t \right] = \mathbb E[X_0] + \tilde{W}_t.$$
Moreover, the equilibrium control follows
\begin{equation}
\label{eq:alpha011p}
	\hat \alpha_t =  - 2 a(t) \left(\hat X_t - \hat \mu_t \right).
\end{equation}

\end{prop}

\begin{prop}
\label{p:ABCexist}
Suppose Assumptions \ref{a:asm1}-\ref{a:asm2} hold. For the $N$-player game, the path and the control of player $i$ under the equilibrium are given by
\begin{equation}
\label{eq:XihatNp}
    d \hat X_{it}^{(N)}  = - 2 a^N(t)  \left(\hat X_{it}^{(N)}  -\frac{1}{N-1}\sum_{j \ne i}^N  \hat X_{jt}^{(N)} 
	\right) dt + dW_{it}^{(N)} + d \tilde{W}_t,
\end{equation}
and
$$ \hat{\alpha}_{it}^{(N)} = - 2a^N (t) \left(\hat X_{it}^{(N)} -\frac{1}{N-1}\sum_{j \ne i}^N  \hat X_{jt}^{(N)} \right)$$
respectively for $i = 1, 2, \dots, N$, where
$a^{N}$ is the solution of 
\begin{equation}
\label{eq:a_12p}
 a' -\frac{2(N+1)}{N-1}a^2 + \frac{N-1}{N} k = 0, \quad  a(T)= 0.
\end{equation}
\end{prop}

\subsection{Preliminaries} 

We first recall the convergence rate of empirical measures of i.i.d. sequence provided in Theorem 1 of \cite{FG15} and Theorem 5.8 of \cite{CD18I}. 
\begin{lemma}
\label{l:iidrate} 
Let $d = 1$ or $2$.
Suppose $\{X_i: i \in \mathbb N\}$ is a sequence of $d$ dimensional i.i.d. random variables with $\mathbb E [|X_1|^q] < \infty$ for some $q>4$. 
Then, the empirical measure 
$$\rho^N(X) = \frac 1 N \sum_{i=1}^N \delta_{X_i}$$
satisfies
$$
\mathbb E \left[ \mathbb W_p^p \left( \rho^N(X), \mathcal L(X_1) \right) \right] = 
\begin{cases}
O \left(N^{-1/2} \right), & \hbox{ if } p \in (1,2], \\
O \left(N^{-1/2} \right), & \hbox{ if } p=1, d =1, \\
O \left(N^{-1/2} \ln N \right), & \hbox{ if } p=1, d= 2.
\end{cases}
$$
\end{lemma}

Next, we give the definition of some notations that will be used in the following part. Denote $C_b(\mathbb R^d)$ to be the collection of bounded and continuous functions on $\mathbb R^d$, and let $C^1_b(\mathbb R^d) \subset C_b(\mathbb R^d)$ be the space of functions on $\mathbb R^d$ whose first order derivative is also bounded and continuous.

\begin{lemma}
\label{l:joint}
Suppose $m_1, m_2$ are two probability measures on $\mathcal B(\mathbb R^d)$ and $f\in C_b^1(\mathbb R^d, \mathbb R)$, where $\mathcal B(\mathbb R^d)$ is the Borel set on $\mathbb R^d$. Then, 
$$\mathbb W_p(f_* m_1, f_* m_2) \le |Df|_0 \mathbb W_p(m_1, m_2),$$
where $f_* m_j$ is the pushforward measure for $j = 1, 2$, and $|Df|_0 = \sup_{x \in \mathbb R^d} \max\{|\partial_{x_i} f (x)|: i = 1, 2, \dots, d\}.$
\end{lemma}
\begin{proof}
We define a function $F(x, y) = (f(x), f(y)): \mathbb R^{2d} \mapsto \mathbb R^2$.
Note that, for any $\pi \in \Pi(m_1, m_2)$, $F_* \pi \in  \Pi(f_* m_1, f_* m_2)$, i.e., 
$$
F_* \Pi(m_1, m_2) \subset \Pi(f_* m_1, f_* m_2).
$$
Therefore, we have the following inequalities:
\begin{equation*}
\begin{aligned}
\mathbb W_p^p(f_*m_1, f_*m_2) & = 
\inf_{\pi' \in \Pi(f_*m_1, f_*m_2)} \int_{\mathbb R^2} |x-y|^p \pi'(dx, dy) \\ 
& \leq \inf_{\pi' \in  F_* \Pi(m_1, m_2) } \int_{\mathbb R^2} |x-y|^p \pi'(dx, dy) \\ & = \inf_{\pi \in \Pi(m_1, m_2)} \int_{\mathbb R^{2d}} |f(x)-f(y)|^p \pi(dx, dy) \\ 
& \leq |Df|_0^p \inf_{\pi \in \Pi(m_1, m_2)} \int_{\mathbb R^{2d}} |x-y|^p \pi(dx, dy) \\ 
& = |Df|_0^p \mathbb W_p^p(m_1, m_2).
\end{aligned}
\end{equation*}
\end{proof}

\begin{lemma}
\label{l:lln2d}
Let $\{X_i: i \in \mathbb N\}$ be a sequence of $d$ dimensional random variables in $(\Omega, \mathcal F, \mathbb P)$. 
Let $f\in C_b^1(\mathbb R^d)$. We also denote by $f(X)$ the sequence $\{f(X_i): i \in \mathbb N\}$. Then 
$$\mathbb W_p \left( \rho^N(f(X)), \mathcal L(f(X_1)) \right) \le |D f|_0  \mathbb W_p \left( \rho^N(X), \mathcal L(X_1) \right), \ \hbox{ almost surely}$$
where $|Df|_0 = \sup_{x \in \mathbb R^d} \max\{|\partial_{x_i} f (x)|: i = 1, 2, \dots, d\}.$
\end{lemma}
\begin{proof}
For any sequence $\{c_i: i \in \mathbb N\}$ in $\mathbb R^d$, the empirical measure
$\rho^N(c) := \frac 1 N \sum_{i=1}^N \delta_{c_i}$
satisfies
$$\rho^N(f(c)) = f_* \rho^N(c),$$
since
$$
\langle \phi, \rho^N(f(c)) \rangle = \frac 1 N \sum_{i=1}^N \phi (f (c_i)) = \langle \phi \circ f, \rho^N(c)\rangle, \quad \forall \phi\in C_b(\mathbb R^d).$$
This implies that 
$$\rho^N(f(X)) = f_*\rho^N(X), \ \hbox{ almost surely}.$$
On the other hand, we also have
$$\mathcal L(f(X_1)) (A) = \mathbb P(f(X_1) \in A) = \mathbb P(X_1 \in f^{-1}(A)) = f_* \mathcal L(X_1) (A), \quad \forall A \in \mathcal B(\mathbb R^d).$$
Therefore, the conclusion follows by applying Lemma \ref{l:joint}.
\end{proof}

\subsection{Empirical measures of a sequence with a common noise}
We are going to apply lemmas from the previous subsection to study the convergence of empirical measures of a sequence with a common noise in the following sense.
\begin{defi} \label{d:com1}
We say a sequence of random variables $X = \{X_i: i \in \mathbb N\}$ is a sequence with a common noise, if there exists a random variable $\beta$ such that
\begin{itemize}
\item
 $X-\beta = \{X_i - \beta: i \in \mathbb N\}$ is a sequence of i.i.d. random variables,
 \item 
 $\beta$ is independent to $X-\beta$.
\end{itemize}
\end{defi}
By this definition,  a sequence with a common noise is i.i.d. if and only if $\beta$ is a deterministic constant.

\begin{example}\label{e:01}
Let $q>4$ be a given constant and $X = \{X_i: i \in \mathbb N\}$ be a $1$-dimensional sequence of $L^q$ random variables with a common noise term $\beta$, 
where 
$$X_i - \beta = \gamma_i + \sigma \alpha_i.$$
In above, $\{(\alpha_i, \gamma_i): i \in \mathbb N\}$ is a sequence of $2$-dimensional i.i.d. random variables independent to $\beta$, and $\sigma$ is a given non-negative constant.
Let $\rho^N(X)$ be the empirical measure defined by
$$\rho^N(X) = \frac 1 N \sum_{i=1}^N \delta_{X_i}.$$
\end{example}
The first question is
\begin{itemize}
\item
[(Qa)] In Example \ref{e:01}, where does $\rho^N(X)$ converge to?
\end{itemize}
For any test function $\phi \in C_b (\mathbb R)$,
$$
\langle \phi, \rho^N(X) \rangle \displaystyle = \frac 1 N \sum_{i=1}^N \phi(X_i)
 =  \frac 1 N \sum_{i=1}^N \phi(\gamma_i + \sigma\alpha_i + \beta).
$$
Since $\beta$ is independent to $(\alpha_i, \gamma_i)$, by Example 4.1.5 of 
\cite{Dur05} together with the Law of Large Numbers, we have
$$ \frac 1 N \sum_{i=1}^N \phi(\gamma_i + \sigma\alpha_i  + c) 
\to  \mathbb E[ \phi( \gamma_1 + \sigma\alpha_1 + c)] = 
\mathbb E[ \phi ( \gamma_1 + \sigma\alpha_1 + \beta)|\beta = c],  \quad \forall c\in \mathbb R.$$ 
Therefore, we conclude that
\begin{equation*}
\begin{aligned}
\langle \phi, \rho^N(X) \rangle  & \to \mathbb E [ \phi(\gamma_1 + \sigma\alpha_1 + \beta) | \beta ], \quad \beta-a.s.  \\
& = \langle \phi, \mathcal L(\gamma_1 + \sigma\alpha_1 + \beta | \beta) \rangle, \quad \beta-a.s. \\
\end{aligned}
\end{equation*}
Hence, the answer for the (Qa) is
\begin{itemize}
\item
$\rho^N(X) \Rightarrow \mathcal L(X_1 | \beta)$, $\beta$-a.s. 
More precisely, since all random variables are square-integrable, the weak convergence implies, for all $p\in [1,2]$, 
$$\mathbb W_p \left(\rho^N(X), \mathcal L \left(X_1 |\beta \right) \right) \to 0, \quad \beta-a.s.$$
\end{itemize}
The next question is 
\begin{itemize}
\item
[(Qb)] In Example \ref{e:01}, what's the convergence rate in the sense 
$\mathbb E \left[ \mathbb W_p^p \left(\rho^N(X), \mathcal L \left(X_1 |\beta \right) \right) \right]$?
\end{itemize}
 Since $\beta$ is independent to $\gamma_1 + \sigma\alpha_1 $, by Example 4.1.5 of 
\cite{Dur05}, we have
$$\mathbb E[ \phi(\gamma_1 + \sigma\alpha_1  + \beta)|\beta = c]  
= \mathbb E[ \phi(\gamma_1 + \sigma\alpha_1  + c)], \quad 
\forall \phi\in C_b (\mathbb R), c\in \mathbb R,$$ 
or equivalently, if one takes $c = \beta(\omega)$,
$$\mathcal L(X_1|\beta) (\omega) = 
\mathcal L(\gamma_1 + \sigma\alpha_1  + \beta| \beta) (\omega)
= \mathcal L(\gamma_1 + \sigma\alpha_1  +c).$$
On the other hand, with $c = \beta (\omega)$, 
$$
\rho^N(X) (\omega) = \rho^N(X (\omega)) 
= \frac 1 N \sum_{i=1}^N \delta_{\gamma_i(\omega) + \sigma\alpha_i (\omega) + c}.
$$
From the above two identities, with $c = \beta(\omega)$, we can write
\begin{equation}
\label{eq:was01}
\mathbb W_p \left(\rho^N(X)(\omega), \mathcal L(X_1  | \beta = c) (\omega) \right) =  \mathbb W_p \left( 
\frac 1 N \sum_{i=1}^N \delta_{\gamma_i(\omega) + \sigma\alpha_i (\omega) + c},  \mathcal L(\gamma_1 + \sigma\alpha_1  + c)
\right).
\end{equation}
Now we can conclude (Qb) in the next lemma.
\begin{lemma}
\label{l:e01}
Let $p\in [1, 2]$ be a given constant.
For a sequence $X = \{X_i: i \in \mathbb N\}$ with a common noise $\beta$ as of Example \ref{e:01}, we have
$$\mathbb E \left[ \mathbb W_p^p \left( \rho^N(X), \mathcal L( X_1  | \beta) \right) \right] = 
O \left(N^{-\frac{1}{2}} \right).$$
\end{lemma}
\begin{proof}
Originally, $X_i =\gamma_i + \sigma\alpha_i   + \beta$  of Example \ref{e:01} are dependent due to the common term $\beta$.  
We apply \eqref{eq:lwas1_2} in Lemma \ref{l:was1} in Appendix to \eqref{eq:was01} and obtain
\begin{equation*}
\begin{aligned}
\mathbb W_p \left( \rho^N(X)(\omega), \mathcal L(X_1  | \beta) (\omega) \right) 
&=  \mathbb W_p \left( 
\frac 1 N \sum_{i=1}^N 
\delta_{\gamma_i(\omega) + \sigma\alpha_i (\omega) + \beta (\omega)},  
\mathcal L(\gamma_1 + \sigma\alpha_1 + \beta(\omega))
\right) \\ 
& = \mathbb W_p \left( \rho^N(\gamma(\omega) + \sigma\alpha (\omega)),  \mathcal L(\gamma_1 + \sigma\alpha_1 ) \right).
\end{aligned}
\end{equation*}
Now, the convergence of empirical measures is equivalent to the ones of i.i.d. sequence 
$\{
\gamma_i + \sigma\alpha_i: i\in \mathbb N
\}$. The conclusion follows from Lemma \ref{l:iidrate}.
\end{proof}

Next, we present the uniform convergence rate by combining Lemma \ref{l:lln2d}.
\begin{lemma}
\label{l:unif}
In Example \ref{e:01}, we use $X(\sigma)$ to denote $X$ to emphasize its dependence on $\sigma$. Then,
$$\mathbb E 
\left[
\sup_{\sigma\in [0, 1]} 
\mathbb W_p^p \left(\rho^N(X(\sigma)), \mathcal L \left(X_1(\sigma) |\beta \right) \right) 
\right]
= 
\begin{cases}
O \left(N^{-\frac{1}{2}} \ln (N) \right), & \hbox{ if } p=1, \\
O\left(N^{-\frac{1}{2}} \right), & \hbox{ if } 1< p \leq 2.
\end{cases}
$$
\end{lemma}
\begin{proof}
Note that, by  \eqref{eq:lwas1_2} in Lemma \ref{l:was1} in Appendix,
$$\mathbb W_p^p \left(\rho^N(X(\sigma)), \mathcal L \left(X_1(\sigma) |\beta \right) \right) = 
\mathbb W_p^p \left(\rho^N(\gamma_i + \sigma \alpha_i), \mathcal L \left(\gamma_1 + \sigma \alpha_1  \right) \right).
$$
Next, applying Lemma \ref{l:lln2d} with $f(x, y) = x + \sigma y$, we obtain
\begin{equation*}
\begin{aligned}
\sup_{\sigma \in [0, 1]} \mathbb W_p^p \left(\rho^N(\gamma_i + \sigma \alpha_i), \mathcal L \left(\gamma_1 + \sigma \alpha_1  \right) \right)
& \leq 
\sup_{\sigma \in [0, 1]} \max\{1, \sigma^p\} \mathbb W_p^p \left(\rho^N((\gamma,  \alpha)), \mathcal L \left((\gamma_1, \alpha_1)  \right) \right)
\\ 
& =  \mathbb W_p^p \left(\rho^N((\gamma,  \alpha)), \mathcal L \left((\gamma_1, \alpha_1)  \right) \right).  
\end{aligned}
\end{equation*}
At last, using Lemma \ref{l:iidrate} for the $2$-dimensional i.i.d. sequence $\{(\gamma_i, \alpha_i): i\in \mathbb N\}$, we obtain the desired conclusion.
\end{proof}

\subsection{Generalization of the convergence to triangular arrays}
Unfortunately, $(\hat X_{1t}^{(N)}, \hat X_{2t}^{(N)}, \ldots, \hat X_{Nt}^{(N)})$ of the $N$-player's game does not have a clean structure with a common noise term $\beta$ given in Example \ref{e:01}. Therefore, we need  a generalization of the convergence result in Example \ref{e:01} to a triangular array. To proceed, we provide the following lemma.

\begin{lemma}
\label{l:triangle}
Let $\sigma>0$, $q>4$, and
$$X_i^N (\sigma) = \gamma_i^N + \sigma \alpha_i^N + \Delta_i^N (\sigma)+ \beta, 
\hbox{ and } \hat X (\sigma) = \hat \gamma + \sigma \hat \alpha  + \beta,$$
where 
\begin{itemize}
\item $(\gamma^N, \alpha^N) = \{ (\gamma_i^N, \alpha_i^N): i \in \mathbb N \}$ 
is a sequence of $2$-dimensional i.i.d. random variables with distribution identical to 
$\mathcal L((\hat \gamma, \hat\alpha))$ with $(\hat \gamma, \hat \alpha) \in L^q$ for some $q > 4$,
\item $\beta \in L^q$ is independent to the random variables $(\gamma_i^N, \alpha_i^N, \hat \gamma, \hat\alpha)$,
\item $\displaystyle \max_{i=1, 2, \dots, N} \mathbb E \left[ \sup_{\sigma\in [0,1]} |\Delta_i^N (\sigma) |^2 \right]=O (N^{-1})$.
\end{itemize}
Let $\rho^N(X^N)$ be the empirical measure given by
$$\rho^N(X^N) = \frac 1 N \sum_{i=1}^N \delta_{X_i^N}.$$
Then, we have the following three results: For $p\in [1,2]$,
\begin{equation}
\label{eq:ltri1}
\mathbb W_p \left(\mathcal L \left(X_1^N (\sigma) \right), \mathcal L \left(\hat X(\sigma) \right) \right) = O \left(N^{-\frac{1}{2}} \right),
\end{equation}
\begin{equation}
\label{eq:ltri2}
\sup_{\sigma\in [0,1]} 
\mathbb E\left[ \mathbb W_p^p \left( \rho^N \left(X^N (\sigma) \right), 
\mathcal L \left(\left. \hat X(\sigma)  \right \vert \beta \right) \right) \right] = O \left( N^{-\frac{1}{2}} \right),
\end{equation}
and
\begin{equation}
\label{eq:ltri3}
\mathbb E\left[ \sup_{\sigma\in [0,1]}  \mathbb W_p^p \left( \rho^N \left(X^N (\sigma) \right), 
\mathcal L \left(\left. \hat X(\sigma)  \right \vert \beta \right) \right) \right] =
\begin{cases}
O \left(N^{-\frac{1}{2}} \ln (N) \right), & \hbox{ if } p=1, \\
O\left(N^{-\frac{1}{2}} \right), & \hbox{ if } p > 1.
\end{cases}
\end{equation}
\end{lemma}

\begin{proof}
We will omit the dependence of $\sigma$  if there is no confusion, for instance, we use $X$ in lieu of $X(\sigma)$.
Since $\mathcal L(\hat X) = \mathcal L(X_1^N - \Delta_1^N)$, 
the first result \eqref{eq:ltri1} directly follows from 
$$
 \mathbb W_p^p \left(\mathcal L \left( X_1^N \right), \mathcal L \left(\hat X \right) \right) 
 \le 
\mathbb E \left[ \left\vert \Delta_1^N \right\vert^p  \right]
\le \left( \mathbb E \left[ \left\vert \Delta_1^N \right\vert^2 \right] \right)^{\frac{p}{2}} = O \left(N^{-\frac{p}{2}} \right).
 $$
Next, we set $Y_i^N(\sigma) = \gamma_i^N + \sigma \alpha_i^N + \beta$. 
By the definition of empirical measures, we have
\begin{equation}
\label{eq:no1}
\mathbb W_p^p \left( \rho^N \left(X^N \right), \rho^N \left(Y^N \right) \right) \le \frac{1}{N} \sum_{i=1}^N  \left\vert X_i^N - Y_i^N \right\vert^p 
 = \frac{1}{N} \sum_{i=1}^N  \left\vert \Delta_i^N \right\vert^p.
\end{equation}
From the third condition on $\Delta_i^N$, we obtain
$$
\mathbb E \left[ \mathbb W_p^p \left( \rho^N \left(X^N \right), \rho^N \left( Y^N \right) \right) \right] = O \left(N^{-\frac{p}{2}} \right).
$$
By Lemma \ref{l:e01}, we also have
$$
\mathbb E \left[ \mathbb W_p^p \left( \rho^N \left( Y^N \right), \mathcal L \left( \left. \hat X \right\vert \beta \right) \right) \right] = O \left(N^{-\frac{1}{2}} \right).
$$
In the end, \eqref{eq:ltri2} follows from the triangle inequality together with the fact that $p \geq 1$.
Finally, for the proof of \eqref{eq:ltri3}, we first use \eqref{eq:no1} to write
$$
\mathbb W_p^p \left( \rho^N \left(X^N (\sigma) \right), 
\mathcal L \left(\left. \hat X(\sigma)  \right \vert \beta \right) \right) \leq 
2^{p-1} \left(
\mathbb W_p^p \left( \rho^N \left(Y^N (\sigma) \right), 
\mathcal L \left(\left. \hat X(\sigma)  \right \vert \beta \right) \right) +  \frac{1}{N} \sum_{i=1}^N  \left\vert \Delta_i^N(\sigma) \right\vert^p
\right).
$$
Applying Lemma \ref{l:unif} and the third condition on $\Delta_i^N(\sigma)$, we can conclude  \eqref{eq:ltri3}. 
\end{proof}

\subsection{Proof of Theorem \ref{t:main1}} \label{s:pthm}
For simplicity, let us introduce the following notations:
$$\mathcal E_t(a) = \exp \left \{ \int_0^t a(s) ds \right \},  \quad \mathcal E_t(a, M) = \int_0^t \mathcal E_s(a) dM_s$$
for a deterministic function $a(\cdot)$ and a martingale $M = \{M_t\}_{t \geq 0}$. With these notations, one can write the solution to the Ornstein–Uhlenbeck process
$$d X_t = - a_t X_t dt + dM_t$$
for a determinant function $a$ in the form of
\begin{equation}
\label{eq:ou1}
\mathcal E_t(a) X_t = X_0 + \mathcal E_t(a, M).
\end{equation}

For MFG equilibrium, we define
$$\tilde X_t = \hat X_t - \hat \mu_t.$$
According to \eqref{eq:Xhat03p} in Proposition \ref{p:mfg}, $\tilde X$ satisfies the following equation:
$$\tilde X_t = \tilde X_0 - \int_0^t 2 a_s \tilde X_s ds + W_t.$$
Next, we express the solution of the above SDE in the form of
$$\tilde Y_t := \mathcal E_t(2a) \tilde X_t = \tilde X_0 + \mathcal E_t(2a, W).$$
Note that $\tilde Y$ and $\hat \mu$ are independent. Therefore, $\hat X$ admits a decomposition of 
two independent processes as
$$\hat X_t = \tilde X_t + \hat \mu_t.$$
Furthermore, we have
$$\hat Y_t := \mathcal E_t(2a) \hat X_t = \tilde X_0 +  \mathcal E_t(2a, W) +  \mathcal E_t(2a) \left(\hat \mu_0 + \tilde W_t \right). $$

In the $N$-player game, we define the following quantities:
$$\bar X^{(N)}_t = \frac{1}{N} \sum_{i=1}^N \hat X^{(N)}_{it}, \quad  \bar W^{(N)}_t = \frac{1}{N} \sum_{i=1}^N W^{(N)}_{it},$$
and
$$\tilde X^{(N)}_{it} = \hat X^{(N)}_{it} - \bar X^{(N)}_t.$$
It is worth noting that, by Proposition \ref{p:ABCexist}, we have
\begin{equation*}
\hat X_{it}^{(N)}  = \hat X_{i0}^{(N)} - \int_0^t 2 \frac{N}{N-1} a^N(s) \left(\hat X_{is}^{(N)} - \frac{1}{N} \sum_{j = 1}^N  \hat X_{js}^{(N)} \right) ds + W_{it}^{(N)} + \tilde{W}_t
\end{equation*}
for all $i = 1, 2, \dots, N$, then the mean-field term satisfies
$$\bar X^{(N)}_t = \bar X^{(N)}_0 + \bar W^{(N)}_t  + \tilde W_t$$
and the $i$-th player's path deviated from the mean-field path can be rewritten by
$$
\tilde X^{(N)}_{it} = \tilde X^{(N)}_{i0} - \int_0^t 2  \hat a^N(s) \tilde X^{(N)}_{is} ds + W_{it}^{(N)} - \bar W_t^{(N)},
$$
where 
$$\hat a^N = \frac{N}{N-1} a^N.$$
Next, we introduce
$$\hat Y^{(N)}_{it} = \mathcal E_t \left( 2 \hat a^N \right) \hat X^{(N)}_{it}, \quad 
\tilde Y^{(N)}_{it} = \mathcal E_t \left(2 \hat a^N \right) \tilde X^{(N)}_{it}, \quad
\bar Y^{(N)}_{t} = \mathcal E_t \left(2 \hat a^N \right) \bar X^{(N)}_{t}.$$
Consequently, we obtain the following relationships:
$$\tilde Y^{(N)}_{it} = \tilde X^{(N)}_{i0} + \mathcal E_t \left(2 \hat a^N,  W^{(N)}_{i} - \bar W^{(N)} \right),$$
$$
\bar Y^{(N)}_{t} = \mathcal E_t \left(2 \hat a^N \right) \left(\bar W^{(N)}_t+ \tilde W_t + \bar X^{(N)}_0 \right),
$$
and 
$$\hat Y^{(N)}_{it} = \bar Y^{(N)}_{it} +\tilde Y^{(N)}_{it}.$$

To compare the process $\hat Y^{(N)}_{it}$ with the target process 
\begin{equation}
\label{eq:Y1}
\begin{aligned}
\hat Y_t &= \tilde X_0 +  \mathcal E_t \left(2a, W \right) +  \mathcal E_t(2a) \left(\hat \mu_0 + \tilde W_t \right) \\ 
& = \tilde X_0 +  \sigma_t Z_t + \mathcal E_t(2a) \left(\hat \mu_0 + \tilde W_t \right),
\end{aligned}
\end{equation}
where 
$$\sigma_t = \left( \int_0^t \mathcal E_s(4a) ds\right)^{1/2},$$
and
$$Z_t = \sigma_t^{-1} \mathcal E_t \left(2a, W \right) \sim \mathcal N(0,1),$$
we write $\hat Y^{(N)}_{it}$ by
\begin{equation}
\label{eq:YN1}
\begin{aligned}
\hat Y^{(N)}_{it} &= \tilde X^{(N)}_{i0} + \mathcal E_t \left(2 a,  W^{(N)}_{i} \right) + \Delta^{(N)}_{it} +  \mathcal E_t(2a) \left(\hat \mu_0 + \tilde W_t \right) \\
&= \tilde X^{(N)}_{i0} + \sigma_t Z_{it}^{(N)}+ \Delta^{(N)}_{it} +  \mathcal E_t(2a) \left(\hat \mu_0 + \tilde W_t \right),
\end{aligned}
\end{equation}
where 
$$Z_{it}^{(N)} = \sigma_t^{-1} \mathcal E_t \left(2a, W_i^{(N)} \right) \sim \mathcal N(0,1), $$
and
\begin{equation}
\label{eq:delta1}
\begin{aligned}
\Delta^{(N)}_{it}  &= \left( \mathcal E_t \left(2 \hat a^N, W^{(N)}_{i} \right)  -  \mathcal E_t \left(2 a, W^{(N)}_{i} \right) \right) \\
& \hspace{0.3in} - \mathcal E_t \left(2 \hat a^N, \bar W^{(N)} \right) \\
& \hspace{0.3in} + \left(\mathcal E_t \left( 2 \hat a^N \right) - \mathcal E_t(2 a) \right) \left(\hat \mu_0 + \tilde W_t \right) \\
& \hspace{0.3in} + \mathcal E_t \left(2 \hat a^N \right) \left( \bar X_0^{(N)} - \hat \mu_0 + \bar W_t^{(N)} \right) \\
& := I^{(N)}_{it}+ II^{(N)}_{t} + III^{(N)}_{t} + IV^{(N)}_{t}.
\end{aligned}
\end{equation}
To apply Lemma \ref{l:triangle} to the processes of \eqref{eq:YN1} and \eqref{eq:Y1}, we only need to show the second moment on $\sup_{t\in [0, T]}\Delta^{(N)}_{it}$ is $O(N^{-1})$ for each $i = 1, 2, \dots, N$.
In the following analysis, we will utilize the explicit solution of the ODE:
\begin{itemize}
\item
Let $c, d>0$ be two constants. The solution of 
$$v'(t) - c^2 v^2(t) + d^2 = 0, \quad v(T) = 0$$
is
\begin{equation}
\label{eq:odes1}
v(t) = \frac{d}{c} \cdot \frac{1 - e^{2dc(t-T)}}{1 + e^{2dc(t-T)}}.
\end{equation}
\end{itemize}
We will employ this solution to derive the second-moment estimations of $\sup_{t \in [0, T]} \Delta^{(N)}_{it}$.
\begin{enumerate}
\item From
\eqref{eq:odes1}, we have an estimation of
\begin{equation}
\left\vert a^N (t)- a(t) \right\vert =  \frac{k|T-t|}{N} + o\left({N^{-1}} \right).
\end{equation}
Therefore, we have
\begin{equation}
\label{eq:est1}
\left\vert \mathcal E_t(2 \hat a ^N) - \mathcal E_t(2 a) \right\vert = \frac{2t(T-t)}{N} + o \left(N^{-1} \right)
\end{equation}
and thus by Burkholder-Davis-Gundy (BDG) inequality
\begin{equation*}
\begin{aligned}
\mathbb E \left[ \sup_{t\in [0, T]} \left( I^{(N)}_{it} \right)^2 \right] & = \mathbb E 
\left[ \sup_{t\in [0, T]} \left( \int_0^t \left(\mathcal E_s \left(2 \hat a^N \right) - \mathcal E_s \left( 2 a \right) \right) dW^{(N)}_{is} \right)^2 \right] \\
& \leq C \mathbb E \left[ \left( \int_0^T \left(\mathcal E_s \left(2 \hat a^N \right) - \mathcal E_s \left( 2 a \right) \right) dW^{(N)}_{is} \right)^2  \right] \ \hbox{ for some constant } C>0 \\
& = C \int_0^{T} \left(\mathcal E_s \left(2 \hat a^N \right) - \mathcal E_s \left( 2 a \right) \right)^2 ds \\
& = O \left(N^{-2} \right).
\end{aligned}
\end{equation*}

\item Since $\hat a^N$ is uniformly bounded by $\sqrt{k/2}$, 
$II_t^{(N)}$ is a martingale with its quadratic variance 
$$[II^{(N)}]_T = \frac 1 N \int_0^T \mathcal E_s (4 \hat a^N) ds = O \left(N^{-1} \right).$$
So, we have
$$\mathbb E \left[ \sup_{t\in [0, T]} \left( II^{(N)}_{t} \right)^2 \right] = O \left( N^{-1} \right). $$
\item
From the estimation \eqref{eq:est1}, we also have
$$\mathbb E \left[ \sup_{t\in [0, T]} \left( III^{(N)}_{t} \right)^2 \right] = O \left(N^{-2} \right). $$
\item
By the assumption of i.i.d. initial states, we have
$$
\mathbb E \left[ \sup_{t\in [0, T]}\left( IV^{(N)}_{t} \right)^2 \right] =
\mathcal E_T \left(4 \hat a^N \right) 
\left(Var\left(\bar X_{0}^{(N)} \right) + 
\mathbb E \left[ \sup_{t\in [0, T]} \left( \bar W_t^{(N)} \right)^2 \right] \right) = O \left(N^{-1} \right). $$
\end{enumerate}
As a result, we have the following expression:
\begin{equation}
\label{eq:est2}
\mathbb E \left[ \sup_{t\in [0, T]} \left(\Delta^{(N)}_{it} \right)^2 \right]= O \left( N^{-1} \right), \quad  \forall i = 1, 2, \ldots, N.
\end{equation}
By combining equations \eqref{eq:Y1}, \eqref{eq:YN1}, and \eqref{eq:est2}, we can conclude Theorem \ref{t:main1} by applying Lemma \ref{l:triangle}.


\section{Proposition \ref{p:mfg}: Derivation of the MFG path}
\label{s:section3}
	
This section is dedicated to proving Proposition \ref{p:mfg}, which provides insights into the MFG solution. 
To proceed, in Subsection \ref{s:overview}, we begin by reformulating the MFG problem, assuming a Markovian structure for the equilibrium.
Then, in Subsection \ref{s:Riccati system}, we solve the underlying control problem and derive the corresponding Riccati system.
Finally, in Subsection \ref{s:proof_of_main_result}, we examine the fixed-point condition of the MFG problem, leading to the conclusion.

\subsection{Reformulation}
\label{s:overview}
To determine the equilibrium measure, as defined in Definition \ref{d:neN}, one needs to explore the infinite-dimensional space of random measure flows $m:(0, T] \times \Omega \to \mathcal P_2(\mathbb R)$ until a measure flow satisfies the fixed-point condition $m_t = \mathcal L (\hat X_t | \tilde{\mathcal F}_t)$ for all $t\in(0,T]$, as illustrated in Figure \ref{fig:MFG1}.
		
The first observation is that the cost function $F$ in \eqref{eq:running cost} is only dependent on the measure $m$ through the first two moments with the quadratic cost structure, which is given by
\begin{equation*}
    F(x, m) = k (x^2 - 2x [m]_1 + [m]_2). 
\end{equation*}
Consequently, the underlying stochastic control problem for MFG can be entirely determined by the input given by the $\mathbb R^2$ valued random processes $\mu_t = [m_t]_1$ and $\nu_t = [m_t]_2$, which implies that the fixed point condition can be effectively reduced to merely checking two conditions:
$$\mu_t = \mathbb E \left[\left.\hat X_t \right \rvert \tilde{\mathcal F}_t \right], \ \nu_t = \mathbb E \left[ \left. \hat X_t^2 \right \rvert \tilde{\mathcal F}_t \right].$$
This observation effectively reduces our search from the space of random measure-valued processes $m: (0, T]\times \Omega \mapsto \mathcal P_2(\mathbb R)$ to the space of $\mathbb R^2$-valued random processes $(\mu, \nu): (0, T]\times \Omega \mapsto \mathbb R^2$. 

It is important to note that if the underlying MFG does not involve common noise, the aforementioned observation is adequate to transform the original infinite-dimensional MFG into a finite-dimensional system. In this case, the moment processes $(\mu, \nu)$ become deterministic mappings $[0, T] \to \mathbb{R}^2$. However, the following example demonstrates that this is not applicable to MFG with common noise, which presents a significant drawback in characterizing LQG-MFG using a finite-dimensional system.

\begin{example}\label{exm:1}
To illustrate this point, let's consider the following uncontrolled mean field dynamics:
Let the mean field term $\mu_t := \mathbb{E} [\hat X_t | \tilde{\mathcal{F}}_t]$, where the underlying dynamic is given by
	$$d \hat X_{t} = - \mu_t \tilde{W}_t dt + dW_{t} + d \tilde{W}_t, \quad \hat X_0 = X_0.$$
Here are two key observations:
	\begin{itemize}
		\item $\mu_t$ is path dependent on entire path of $\tilde{W}$, i.e.,
		$$\mu_t = \mu_0 e^{- \int_0^t \tilde{W}_s ds} + e^{- \int_0^t \tilde{W}_s ds} \int_0^t e^{\int_0^s \tilde{W}_r dr} d\tilde{W}_s.$$
		This implies that  
		the $(t, \tilde{W}) \mapsto \mu_t$ is a function on an infinite dimensional domain.
		\item  $\mu_t $ is {\it Markovian}, i.e.,
		$$ d \mu_t = -  \mu_t \tilde{W}_t dt + d \tilde{W}_t.$$
		It is possible to 
		express the $\mu_t$ via a SDE with finite-dimensional coefficient functions of  $(t, \mu_t)$.
    \end{itemize}
\end{example}

To make the previous idea more concrete, we propose the assumption of a Markovian structure for the first and second moments of the MFG equilibrium. In other words, we restrict our search for equilibrium to a smaller space $\mathcal{M}$ of measure flows that capture the Markovian structure of the first and second moments.
\begin{defi}
\label{d:markov}
The space $\mathcal M$ is the collection of all $\tilde{\mathcal F}_t$-adapted measure flows $m: [0,T]\times \Omega \mapsto \mathcal P_2(\mathbb R)$, whose first moment  $[m_t]_1 := \mu_t$ 
and second moment $[ m_t]_2 := \nu_t$ satisfy a system of SDE
	\begin{equation}
		\label{eq:mu_sigma}
		\begin{aligned}
			\displaystyle &\mu_t = \mu_0 + \int_0^t \left(w_1(s)\mu_s +w_2(s)\right) ds + \tilde{W}_t, \\
			\displaystyle & \nu_t = \nu_0 + \int_0^t \left(w_3(s)\mu_s+w_4(s)\nu_s +w_5(s)\mu_s^2 + w_6(s)\right) ds + 2 \int_0^t \mu_s  d \tilde{W}_s, 
		\end{aligned}
	\end{equation}
for some 	smooth deterministic functions $(w_i: i = 1, 2, \ldots, 6)$ for all $t \in [0, T]$.
\end{defi}

\begin{figure}[h]
	\centering
	\includegraphics[width= .6\textwidth]{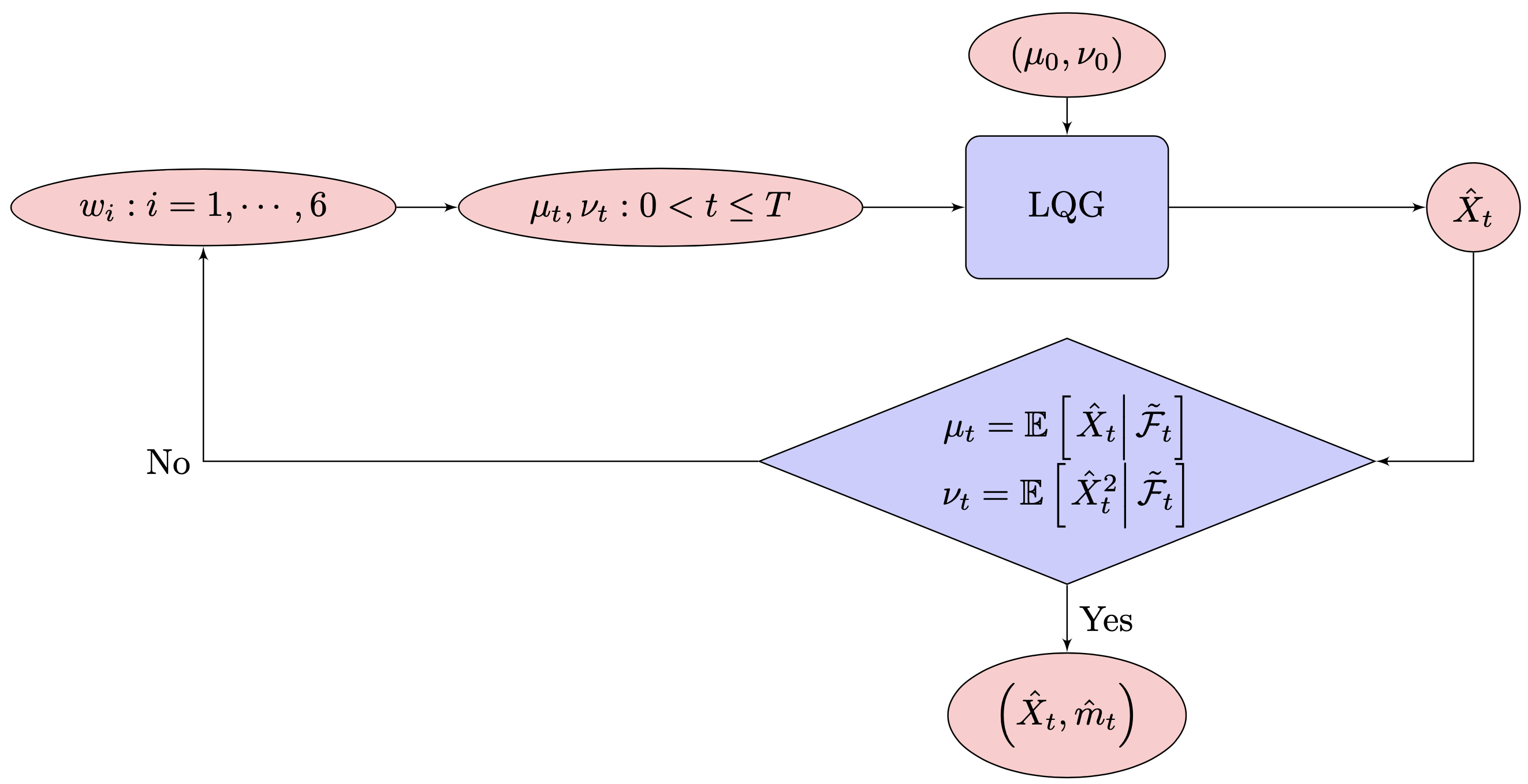}
	\caption{ Equivalent MFG diagram with $\mu_0 = [m_0]_1$ and $\nu_0 = [m_0]_2$.}
	\label{fig:MFG2}
\end{figure}
	
The MFG problem originally given by Definition \ref{d:ne} can be recast as the following combination of stochastic control problem and fixed point condition:
\begin{itemize}
	\item RLQG(Revised LQG): 
	
	Given smooth functions $w = (w_i: i = 1, 2, \ldots, 6)$, we want to find the value function  
	$\bar V = \bar V[w]: [0, T] \times \mathbb R^3 \to \mathbb R$ and optimal path $(\hat X, \hat \mu, \hat \nu)[w]$ from the following control problem:
	\begin{equation*}
	\bar V(t, x, \bar \mu, \bar \nu) = \inf_{\alpha\in \mathcal A}
	\mathbb E \left[ \left.  \int_t^T \left(\frac{1}{2} \alpha_s^2 + \bar F(X_s, \mu_{s}, \nu_s) \right) \ ds \right \rvert X_t = x, \mu_t = \bar \mu, \nu_t = \bar \nu \right]
	\end{equation*}
	with the underlying process  $X$ of \eqref{eq:X} and $(\mu, \nu)$ of \eqref{eq:mu_sigma} and with the cost functions: $\bar F: \mathbb R^3 \mapsto \mathbb R$ given by
	$$ \bar F (x, \bar \mu, \bar \nu) = k (x^2 - 2x \bar \mu + \bar \nu),$$
	where $\bar{\mu}, \bar{\nu}$ are scalars, while $\mu, \nu$ are used as processes.
	\item RFP(Revised fixed point condition):
	
	 Determine $w$ satisfying the following fixed point condition:
\begin{equation}
\label{eq:fpc}
    \hat \mu_s = \mathbb E \left[\left.\hat{X}_s \right\vert \tilde{\mathcal F}_s \right] \hbox{ and } 
    \hat \nu_s = \mathbb E \left[\left.\hat{X}_s^2 \right\vert \tilde{\mathcal F}_s \right], \quad \forall s \in [0, T].
\end{equation}
The equilibrium measure is then $\mathcal N(\hat \mu_t, \hat \nu_t - \hat \mu_t^2)$.

\end{itemize}

\begin{remark}
It is important to highlight that the Markovian structure for the first and second moments of the MFG equilibrium in this manuscript differs significantly from that presented in \cite{JLSY22}. In \cite{JLSY22}, the processes $\mu_t$ and $\nu_t$ are pairs of processes with finite variation, while in our case, they are quadratic variation processes.

Specifically, in \cite{JLSY22}, the coefficient functions depend on the common noise $Y$, whereas in \eqref{eq:mu_sigma}, the coefficient functions $(w_i: i = 1, 2, \dots, 6)$ are independent of the common noise $\tilde{W}$. Instead, the first and second moments of the MFG equilibrium are only influenced by the common noise through an additive term.
\end{remark}
	
\subsection{The generic player's control with a given population measure} 
\label{s:Riccati system}
This section is devoted to the control problem RLQG parameterized by $w$.

\subsubsection{HJB equation}

To simplify the notation, let's denote each function $w_i(t)$ as $w_i$ for $i \in \{1, 2, \ldots, 6\}$. Assuming sufficient regularity conditions, and according to the dynamic programming principle (refer to \cite{Pham09} for more details), the value function $\bar V$ defined in the RLQG problem can be obtained as a solution $v$ of the following Hamilton-Jacobi-Bellman (HJB) equation
\begin{equation*}
\begin{aligned}
\begin{cases}
\vspace{4pt} \displaystyle \partial_t v + \inf_{a \in \mathbb R} \left(a \partial_x v + \frac{1}{2} a^2 \right) + \left(w_{1} \bar \mu+w_{2} \right) \partial_{\bar \mu} v + \left(w_{3} \bar \mu + w_{4} \bar \nu + w_{5} \bar \mu^2 + w_{6} \right) \partial_{\bar \nu} v + \partial_{xx} v + \frac{1}{2} \partial_{\bar \mu \bar \mu} v  \\
\vspace{4pt} \displaystyle \hspace{1in}  + \partial_{x \bar \mu} v + 2\bar \mu^2 \partial_{\bar \nu \bar \nu} v + 2 \bar \mu \partial_{\bar \mu \bar \nu} v + 2 \bar \mu \partial_{x \bar \nu} v + k(x^2 - 2 \bar \mu x + \bar \nu) = 0,  \\
\displaystyle v(T, x, \mu_T,\nu_T) = 0.
\end{cases}
\end{aligned}
\end{equation*}
Therefore, the optimal control has to admit the feedback form of
\begin{equation}
\label{eq:optimal control}
    \hat \alpha(t) = - \partial_x v \left(t, \hat{X_t}, \mu_t,  \nu_t \right),
\end{equation}
and then the HJB equation can be reduced to
\begin{equation}
\label{eq:hjb1}
\begin{aligned}
\begin{cases}
\vspace{4pt} \displaystyle \partial_t v - \frac{1}{2}(\partial_x v)^2 + \left(w_{1} \bar \mu+w_{2} \right) \partial_{\bar \mu} v + \left(w_{3} \bar \mu + w_{4} \bar \nu + w_{5} \bar \mu^2 + w_{6} \right) \partial_{\bar \nu} v + \partial_{xx} v + \frac{1}{2} \partial_{\bar \mu \bar \mu} v  \\
\vspace{4pt} \displaystyle \hspace{1in}  + \partial_{x \bar \mu} v + 2\bar \mu^2 \partial_{\bar \nu \bar \nu} v + 2 \bar \mu \partial_{\bar \mu \bar \nu} v + 2 \bar \mu \partial_{x \bar \nu} v + k(x^2 - 2 \bar \mu x + \bar \nu) = 0,  \\
\displaystyle v(T, x, \mu_T,\nu_T) = 0.
\end{cases}
\end{aligned}
\end{equation}
Next, we identify what conditions are needed for equating the control problem RLQG and the above HJB equation. Denote $\mathcal S$ to be the set of $v$ such that $v \in C^{\infty}$ satisfies
$$\left(1+|x|^2 \right)^{-1} \left(|v|+|\partial_t v| \right) + \left(1+|x| +|\mu| \right)^{-1} \left(|\partial_x v| + |\partial_{\mu} v| \right) + \left( |\partial_{xx} v| + |\partial_{x\mu} v| + |\partial_{\mu \mu} v| + |\partial_{\nu} v| \right) < K$$
for all $(t, x, \mu, \nu)$ for some positive constant $K$.
\begin{lemma}
\label{l:verification}
Consider  the control problem RLQG with some given smooth functions $w = (w_i: i = 1,2, \dots, 6)$.
\begin{enumerate}
\item (Verification theorem) Suppose there exists a solution $v \in {\mathcal S}$ of  \eqref{eq:hjb1}. Then $v(t, x, \bar \mu,\bar \nu) = \bar V(t, x, \bar \mu,\bar \nu)$, and an optimal control is provided by \eqref{eq:optimal control}.
\item Suppose that the value function $\bar V$ belongs to ${\mathcal S}$, and then $\bar V(t, x, \bar \mu,\bar \nu)$ solves HJB equation \eqref{eq:hjb1}. Moreover, $\hat \alpha$ of \eqref{eq:optimal control} is an optimal control.
\end{enumerate}
\end{lemma}
\begin{proof}
\begin{enumerate}
\item First, we prove the verification theorem. Since $v \in {\mathcal S}$, for any admissible $\alpha\in \mathcal H_{\mathbb F}^4$, the process $X^\alpha$ is well defined and one can apply It\^o's formula to obtain
$$\mathbb E \left[v(T, X_T, \mu_T,\nu_T)\right] = v(t, x, \bar \mu,\bar \nu) + \mathbb E \left[ \int_t^T \mathcal G^{\alpha(s)} v(s, X_s, \mu_s, \nu_s) ds \right],$$
where
\begin{equation*}
\begin{aligned}
\mathcal G^a f(s, x, \bar{\mu}, \bar{\nu}) &= \Big(\partial_t + a\partial_x + \partial_{xx} + \left(w_{1}\bar \mu+w_{2}\right) \partial_{\bar\mu} + \left(w_{3}\bar \mu + w_{4}\bar \nu+w_{5}\bar \mu^2 + w_{6}\right) \partial_{\bar\nu}  \\
&  \hspace{0.5in} +\frac{1}{2} \partial_{\bar \mu \bar \mu} + 2 \bar \mu^2 \partial_{\bar \nu \bar \nu} + \partial_{x \bar \mu} + 2 \bar \mu \partial_{\bar \mu \bar \nu} + 2 \bar \mu \partial_{x \bar \nu} \Big) f(s, x, \bar\mu, \bar\nu).
\end{aligned}
\end{equation*}
Note that the HJB equation actually implies that
$$\inf_a \left\{ \mathcal G^a v + \frac 1 2 a^2\right\} = - \bar F,$$
which again yields
$$- \mathcal G^a v \le \frac 1 2 a^2 + \bar F.$$
Hence, we obtain that for all $\alpha\in \mathcal H_{\mathbb F}^4$,
\begin{equation*}
\begin{aligned}
& v(t, x, \bar \mu,\bar \nu) 
\\ 
= \ & \mathbb E \left[ \int_t^T  - \mathcal G^{\alpha(s)} v(s, X_s, \mu_s,\nu_s) ds \right]  + \mathbb E \left[v(T, X_T, \mu_T,\nu_T) \right] 
\\
\le \ & \mathbb E \left[ \int_t^T \left( \frac 1 2 \alpha^2(s) + \bar F(X_s, \mu_s, \nu_s) \right) ds \right] \\
=  \ & J(t, x, \alpha, \bar \mu,\bar \nu).
\end{aligned}
\end{equation*}			
In the above, if $\alpha$ is replaced by $\hat \alpha$ given by the feedback form \eqref{eq:optimal control}, then since $\partial_x v$ is Lipschitz continuous in $x$, there exists corresponding optimal path $\hat X \in \mathcal H_{\mathbb F}^4$. Thus, $\hat \alpha$ is also in $\mathcal H_{\mathbb F}^4$. One can repeat all the above steps by replacing $X$ and $\alpha$ by $\hat X$ and $\hat \alpha$, and $\le $ sign by $=$ sign to conclude that $v$ is indeed the optimal value. 
			
\item
The opposite direction of the verification theorem follows by  taking $\theta\to t$ for the dynamic programming principle, for all stopping time $\theta\in [t, T],$
\begin{equation*}
\begin{aligned}
    & \bar V(t, x, \bar \mu,\bar \nu) \\
	= \ & \mathbb E \left[ \left. \int_t^{\theta} \left(\frac{1}{2} \alpha_s^2 + \bar F(X_s, \mu_s,\nu_s) \right) ds + \bar V(\theta, X_\theta, \mu_\theta, \nu_\theta) \right \vert X_t = x, \mu_t = \bar \mu, \nu_t = \bar \nu  \right], 
\end{aligned}
\end{equation*}
which is valid under our regularity assumptions on all the partial derivatives.
\end{enumerate}
\end{proof}

\subsubsection{LQG solution}
It is worth noting that the costs $\bar F$ of RLQG are quadratic functions in $(x, \bar \mu, \bar \nu)$, while the drift function of the process $\nu$ of \eqref{eq:mu_sigma} is not linear in $(x, \bar \mu, \bar \nu)$. Therefore, the stochastic control problem RLQG does not fit into the typical LQG control structure. Nevertheless, similarly to the LQG solution, we guess the value function to be a quadratic function in the form of
\begin{equation}
\label{eq:v1}
	v(t, x, \bar \mu, \bar \nu) = a(t) x^2 + b(t) \bar \mu^2 +  c(t) \bar \nu + d(t) + e(t) x  + f (t) \bar \mu + g(t) x \bar \mu.
\end{equation}
Under the above setup for the value function $v$, for $t \in [0, T]$, the optimal control is given by
\begin{equation}
\label{eq:alpha}
    \hat \alpha_t = - \partial_x v(t, \hat X_t, \mu_t, \nu_t) = - 2 a(t) \hat X_t - e(t) - g(t) \mu_t,
\end{equation}
and the optimal path $\hat X$ is
\begin{equation}
\label{eq:Xhat}
\begin{cases}
	\vspace{4pt} \displaystyle d \hat X_t = \left( - 2 a(t) \hat X_t - e(t) - g(t) \mu_t \right) dt + dW_t + d \tilde{W}_t, \\
	\displaystyle \hat X_0 = X_0.
\end{cases}
\end{equation}
To proceed, we introduce the following Riccati system of ODEs for $t \in [0, T]$,
\begin{equation}
\label{eq:ode2}
\begin{aligned}
\begin{cases}
    \vspace{4pt}
	\displaystyle a' - 2a^2 + k  = 0, \\
	\vspace{4pt}
	\displaystyle b' - \frac{1}{2} g^2 + 2 b w_1 + c w_5 = 0, \\
	\vspace{4pt}
	\displaystyle c' + c w_4 + k = 0, \\
	\vspace{4pt}
	\displaystyle d' - \frac{1}{2} e^2 + f w_2 + c w_6 + 2a + b + g  = 0, \\
	\vspace{4pt}
    \displaystyle e' -2ae + w_2 g = 0 , \\
    \vspace{4pt}
	\displaystyle f' - eg + w_1 f + 2b w_2 + c w_3 = 0,\\
	\displaystyle g' - 2ag + w_1 g - 2k  = 0,
\end{cases}
\end{aligned}        
\end{equation}
with terminal conditions
\begin{equation}
\label{eq:ode2_terminal}
	a(T) = b(T) = c(T) = d(T) = e(T) = f(T) = g(T) = 0.
\end{equation}
	
\begin{lemma}
\label{l:ricatti1}
	Suppose there exists a unique solution $(a, b, c, d, e, f, g)$ to the Riccati system of ODEs \eqref{eq:ode2}-\eqref{eq:ode2_terminal} on $[0,T]$. Then the value function of (RMFG) is given by
    \begin{equation}
	\label{eq:Vhat}
	\begin{aligned}
			& \bar V (t, x, \bar \mu, \bar \nu) = v(t, x, \bar \mu, \bar \nu) \\ 
			= \ & a(t) x^2 + b(t) \bar \mu^2 +  c(t) \bar \nu + d(t) + e(t) x  + f (t) \bar \mu + g(t) x \bar \mu
	\end{aligned}
    \end{equation}
	for $t \in [0, T]$ and the optimal control and optimal path are given by \eqref{eq:alpha} and \eqref{eq:Xhat}, respectively.
\end{lemma}

\begin{proof}
	With the form of value function $v$ given in \eqref{eq:v1} and the conditional first and second moment of $\hat X_t$ under the $\sigma$-algebra $\tilde{\mathcal F}_t$ given in \eqref{eq:mu_sigma}, we have
	\begin{equation*}
	\begin{aligned}
		&\partial_t v = a'(t) x^2 + e'(t) x + b'(t) \bar \mu^2 + f'(t) \bar \mu + g'(t) x \bar \mu +  c'(t) \bar \nu + d'(t), \\
		& \partial_x v = 2 x a(t) + e(t) + g(t) \bar \mu, \\
		& \partial_{xx} v = 2a(t), \\
		& \partial_{\bar\mu} v = 2b(t) \bar \mu + f(t) + g(t)x, \\
		& \partial_{\bar\nu} v = c(t), \\
		& \partial_{\bar \mu \bar \mu} v = 2b(t), \\
		& \partial_{x \bar \mu} v = g(t), \\ 
		& \partial_{\bar \mu \bar \nu} v = \partial_{\bar \nu \bar \nu} v  = \partial_{x \bar \nu} v = 0.
	\end{aligned}
	\end{equation*}
	Plugging them back to the HJB equation in \eqref{eq:hjb1}, we get a system of ODEs in \eqref{eq:ode2} by equating $x$, $\bar \mu$, $\bar \nu$-like terms in each equation with the terminal conditions given in \eqref{eq:ode2_terminal}.
		
	Therefore, any solution $(a, b, c, d, e, f, g)$ of a system of ODEs \eqref{eq:ode2} leads to the solution of HJB \eqref{eq:hjb1} in the form of the quadratic function given by \eqref{eq:Vhat}. Since the $(a, b, c, d, e, f, g)$ are differentiable functions on the closed set $[0, T]$, they are also bounded, and thus the regularity conditions needed for $v \in \mathcal S$ is valid. Finally, we invoke the verification theorem given by Lemma \ref{l:verification} to conclude the desired result.
\end{proof}

\subsection{Fixed point condition and the proof of Proposition \ref{p:mfg}}
\label{s:proof_of_main_result}

Returning to the ODE  system  \eqref{eq:ode2}, there are $7$ equations, whereas we need to determine a total of $13$ deterministic functions of $[0, T]\times \mathbb R$ to characterize MFG. These are
$$(a, b, c, d, e, f, g) \quad \hbox{ and } \quad  (w_{i}: i = 1, 2, \ldots, 6).$$
In this below, we identify the missing $6$ equations by checking the fixed point condition of RFP. This leads to a complete characterization of the equilibrium for MFG in Definition \ref{d:ne}.

\begin{lemma}
\label{l:fixed_point_condition}
    With the dynamic of the optimal path $\hat{X}$ defined in \eqref{eq:Xhat}, the fixed point condition \eqref{eq:fpc} implies that the first moment $\hat \mu_s := \mathbb E [\hat{X}_s \vert \tilde{\mathcal F}_s ]$ and the second moment $\hat \nu_s := \mathbb E [\hat{X}_s^2 \vert \tilde{\mathcal F}_s ]$ of the optimal path conditioned on $\tilde{\mathcal F}_t$ satisfy
\begin{equation}
\label{eq:mu_sigma_gen}
\begin{aligned}
\begin{cases}
    \vspace{4pt}
	\displaystyle \hat \mu_s = \bar\mu +
	\int_t^s \left( \left( -2a(r) - g(r) \right) \hat \mu_r - e(r)\right) dr + \tilde{W}_s, \\
	\displaystyle \hat \nu_s = \bar\nu + \int_t^s  \left(2-4a(r)  \hat \nu_r - 2e(r) \hat\mu_r - 2g(r)\hat\mu_r^2 \right) dr + \int_t^s 2 \hat \mu_r d \tilde{W}_r,
\end{cases}
\end{aligned}
\end{equation}
for $s\ge t$, and thus the coefficient functions $w = (w_i: i = 1, 2, \dots, 6)$ in \eqref{eq:mu_sigma} satisfy the following equations:
\begin{equation}
\label{eq:w12}
	w_{1} =-2a - g, \ w_{2} = -e, \ w_{3} = -2e, \ w_{4} = -4a, \ w_{5} = -2g, \ w_{6} = 2, \quad \forall t \in [0, T].
\end{equation}  
\end{lemma}
\begin{proof}
    With the dynamic of the optimal path $\hat{X}$ given by \eqref{eq:Xhat}, we have
    $$\hat X_t = X_0 + \int_0^t \left( - 2 a(s) \hat X_s - e(s) - g(s) \hat \mu_s \right) ds + W_t + \tilde{W}_t,$$
    and since the functions $a, e, g$ are continuous on $[0, T]$, then we can change of order of integration and expectation and it yields
    \begin{equation*}
    \begin{aligned}
        \hat \mu_t &= \mathbb E \left[\left. \hat X_t \right\vert \tilde{\mathcal F}_t \right] \\
        & = \mathbb E \left[\left. X_0 \right\vert \tilde{\mathcal F}_t \right] + \int_0^t \left(-2a(s) \hat \mu_s - e(s) - g(s) \hat \mu_s \right) ds +  \mathbb E \left[\left. W_t \right\vert \tilde{\mathcal F}_t \right] +  \mathbb E \left[\left. \tilde{W}_t \right\vert \tilde{\mathcal F}_t \right] \\
        & = \mathbb E \left[\left. X_0 \right\vert \tilde{\mathcal F}_t \right] + \int_0^t \left(-2a(s) \hat \mu_s - e(s) - g(s) \hat \mu_s \right) ds +  \tilde{W}_t.
    \end{aligned}
    \end{equation*}
    Similarly, applying It\^o's formula, we obtain 
    $$\hat X_t^2 = X_0^2 + \int_0^t \left(2 - 4a(s) \hat X_s^2 - 2e(s) \hat X_s - 2g(s) \hat \mu_s \hat X_s \right) ds + \int_0^t 2 \hat X_s d W_s + \int_0^t 2 \hat X_s d \tilde{W}_s,$$
    and it follows that
    \begin{equation*}
    \begin{aligned}
        \hat \nu_t &= \mathbb E \left[\left. \hat X_t^2 \right\vert \tilde{\mathcal F}_t \right] \\
        & = \mathbb E \left[\left. X_0^2 \right\vert \tilde{\mathcal F}_t \right] + \int_0^t \left(2 - 4a(s) \hat \nu_s - 2 e(s) \hat \mu_s - 2g(s) \hat \mu_s^2 \right) ds +  \mathbb E \left[\left. \int_0^t 2 \hat X_s d W_s \right\vert \tilde{\mathcal F}_t \right] +  \mathbb E \left[\left. \int_0^t 2 \hat X_s d \tilde{W}_s \right\vert \tilde{\mathcal F}_t \right] \\
        & = \mathbb E \left[\left. X_0^2 \right\vert \tilde{\mathcal F}_t \right] +\int_0^t \left(2 - 4a(s) \hat \nu_s - 2 e(s) \hat \mu_s - 2g(s) \hat \mu_s^2 \right) ds +  \int_0^t 2 \hat \mu_s d \tilde{W}_s. 
    \end{aligned}
    \end{equation*}
    Thus the desired result in \eqref{eq:mu_sigma_gen} is obtained. Next, comparing the terms in \eqref{eq:mu_sigma} and \eqref{eq:mu_sigma_gen}, to satisfy the fixed point condition in MFG, we require another $6$ equations in \eqref{eq:w12} for the coefficient functions $w = (w_i: i = 1, 2, \dots, 6)$.
\end{proof}

Using further algebraic structures, one can reduce the ODE system of $13$ equations composed by  \eqref{eq:ode2} and \eqref{eq:w12} into a system of $4$ equations.
	
\begin{proof}
[{\bf Proof of Proposition \ref{p:mfg}}]
Let the smooth and bounded functions $\{w_i: i = 1, 2, \dots, 6\}$ be given, the functions $\left(a, b, c, d, e, f, g \right)$ in \eqref{eq:ode2} is a coupled linear system, and thus their existence, uniqueness and boundedness is shown by Theorem 12.1 in \cite{antsaklis2006}.

Plugging the $6$ equations in \eqref{eq:w12} to the ODE system \eqref{eq:ode2}, we obtain
\begin{equation*}
\begin{aligned}
\begin{cases}
    \vspace{4pt}
    \displaystyle a' - 2a^2 + k  = 0, \\
    \vspace{4pt}
    \displaystyle b' - \frac{1}{2} g^2 -4ab -2bg -2cg = 0, \\
    \vspace{4pt}
    \displaystyle c' -4ac + k = 0, \\
    \vspace{4pt}
    \displaystyle d' - \frac{1}{2} e^2 -ef + 2c + 2a + b + g = 0, \\
    \vspace{4pt}
    \displaystyle e' -2ae -eg = 0 , \\
    \vspace{4pt}
	\displaystyle f' - eg -2a f -gf - 2be - 2ce = 0,\\
	\vspace{4pt}
	\displaystyle g' - 4ag - g^2 - 2k  = 0,
\end{cases}
\end{aligned}
\end{equation*}
with the terminal conditions
\begin{equation*}
    a(T) = b(T) = c(T) = d(T) = e(T) = f(T) = g(T) = 0.
\end{equation*}
		
Let $l = 2a + g$, and it easily to obtain
\begin{equation*}
	l'(t) -l^2(t) = 0, \quad l(T) = 0,
\end{equation*}
which implies that $l(t) = 2a(t) + g(t) = 0$ for all $t \in [0, T]$. This gives the result that $g = -2a$ and it yields $e' = 0$. Then with $e(T) = 0$, we have $e(t) = 0$ for all $t \in [0, T]$ and thus one can obtain $f' = 0$, which indicates that $f(t) = 0$ for all $t \in [0, T]$ as $f(T) = 0$. Therefore the ODE system \eqref{eq:ode2} can be simplified to the following form about
$\left(a(t), b(t), c(t), d(t) : t \in [0, T] \right)$:
\begin{equation}
\label{eq:ode1}
\begin{aligned}
\begin{cases}
\vspace{4pt}
\displaystyle a'(t) -2a^2(t) + k = 0, \\
\vspace{4pt}
\displaystyle b'(t) - 2 a^2(t) + 4 a(t) c(t) = 0, \\
\vspace{4pt}
\displaystyle c'(t) -4 a(t) c(t) + k = 0,\\
\displaystyle d'(t) + b(t) + 2c(t) = 0,
\end{cases}
\end{aligned}
\end{equation}
with the terminal conditions
\begin{equation}
\label{eq:terminal conditions}
a(T) = b(T) = c(T) = d(T) = 0.
\end{equation}
The unique solvability of the Riccati system
\eqref{eq:ode1}-\eqref{eq:terminal conditions} is proven in Lemma \ref{l:solution of Riccati system} in the Appendix. Note that the solution $a$ of \eqref{eq:a_odep} is consistent with the solution of the Riccati system given by equations \eqref{eq:ode1}-\eqref{eq:terminal conditions}. 
	
In this case, since $2a + g = 0$ and $e=0$ for all $t \in [0, T]$, it follows that $\hat\mu_s = \bar\mu + \tilde{W}_s$ for all $s \in [t, T]$ from the fixed point result \eqref{eq:mu_sigma_gen}. Similarly,
$$\hat\nu_s = \bar\nu + \int_t^s \left(2 + 4a(r)\hat \mu_r^2 - 4 a(r) \hat\nu_r \right) \, d r + \int_t^s 2 \hat \mu_r \, d \tilde{W}_r, \quad \forall s \in [t, T].$$
Plugging $e = 0$ and $\hat \mu_s = \bar \mu + \tilde{W}_r$ back to \eqref{eq:alpha}, we obtain the optimal control by
$$\hat{\alpha}_s = 2a(s)(\bar \mu + \tilde{W}_s - \hat{X}_s).$$
Moreover, since $e = f = 0$ and $g = -2a$ for $s \in [t, T]$, the value function can be simplified from  \eqref{eq:v1} to
$$v(t, x, \bar\mu, \bar\nu) = a(t) x^2 -2a(t) x \bar \mu + b(t)\bar\mu^2 +  c(t) \bar \nu + d(t).$$
This concludes Proposition \ref{p:mfg}.
		
\end{proof}

\section{The $N$-Player Game}
\label{s:section4}
	
This section focuses on proving Proposition \ref{p:ABCexist} regarding the corresponding $N$-player game. For simplicity, we can omit the superscript $(N)$ when referring to the processes in the sample space $\Omega^{(N)}$.

To begin, we address the $N$-player game in Subsection~\ref{s:raw}, where we solve it and obtain a Riccati system containing $O(N^3)$ equations. Subsequently, we reduce the relevant Riccati system to an ODE system in Subsection~\ref{s:simplify}, which has a dimension independent of $N$. This simplified system forms the fundamental component of the convergence result.

\subsection{Characterization of the $N$-player game by Riccati system}
\label{s:raw}
It is important to emphasize that based on the problem setting in Subsection \ref{s:n-player} and the running cost for each player specified in \eqref{eq:running cost_n player}, the $N$-player game can be classified as an $N$-coupled stochastic LQG problem. As a result, the value function and optimal control for each player can be determined by means of the following Riccati system:

For $i = 1, 2, \ldots, N$, consider
\begin{equation}
\label{eq:ABC}
\begin{cases}
	\vspace{4pt} \displaystyle A_{i}' - 2 A_{i}^\top e_ie_i^\top A_{i} - 4 \sum_{j \ne i}^{N} A_{j}^\top e_je_j^\top  A_{i} + \frac{k}{N} \sum_{j \ne i}^{N} \left(e_i-e_j\right) \left(e_i-e_j\right)^\top= 0,\\
	\vspace{4pt} \displaystyle B_{i}' -2 A_{i}^\top e_i e_i^\top B_{i} - 2 \sum_{j \ne i}^{N} \left(A_{i}^\top e_je_j^\top B_{j}+A_{j}^\top e_j e_j^\top B_{i}\right) = 0,  \\
	\vspace{4pt} \displaystyle C_{i}' -\frac{1}{2} B_{i}^\top e_ie_i^\top B_{i} - \sum_{j \ne i}^{N} B_{j}^\top e_j e_j^\top B_{i} + 2 tr(A_{i}) =0,\\
	\displaystyle A_{i}(T) = B_{i}(T) = C_{i}(T) = 0,
\end{cases}
\end{equation}
where $A_i$ is $N\times N$ symmetric matrix, $B_i$ is $N$-dimensional vector, $C_{i} \in \mathbb R$ is a real constant, and $e_i$ is the $i$-th natural basis in $\mathbb R^N$ for each $i = 1, 2, \dots, N$.
	
\begin{lemma}
\label{l:riccati-N}
Suppose $(A_{i}, B_{i}, C_{i}: i = 1, 2, \ldots, N)$ is the solution of the Riccati system \eqref{eq:ABC}.
Then, the value functions of $N$-player game defined by \eqref{eq:value_i} is
$$V_i \left(x^{(N)} \right) = \left(x^{(N)} \right)^\top A_{i}(0) x^{(N)} + \left(x^{(N)} \right)^\top B_{i}(0) + C_{i}(0), \quad i=1, 2, \ldots, N.$$ 
Moreover, the path and the control under the equilibrium are given by
\begin{equation}
\label{eq:Xihat}
    d \hat X_{it}^{(N)} = \left(-2 (A_{i}(t))_i^\top \hat{X}_t^{(N)} -(B_{i}(t))_i\right) dt + dW^{(N)}_{it} + d \tilde{W}_t,
\end{equation}
and

\begin{equation*}
    \hat{\alpha}^{(N)}_{it}   = -2 (A_{i}(t))_i^\top \hat{X}^{(N)}_t -(B_{i}(t))_i
\end{equation*}
for each $i = 1, 2, \dots, N$, where $(A)_i$ denotes the $i$-th column of matrix $A$, $(B)_i$ denotes the $i$-th entry of vector $B$ and $\hat{X}^{(N)}_t = [\hat X^{(N)}_{1t},  \hat X^{(N)}_{2t},  \dots, \hat X^{(N)}_{Nt} ]^{\top}$.
\end{lemma}

\begin{proof}
From the dynamic programming principle, it is standard that, under enough regularities, the players' value function $V(x^{(N)})= (V_1, V_2, \dots, V_N)(x^{(N)})$ can be lifted to  the solution $v_{i}(t, x^{(N)})$ of the following system of HJB equations, for $i = 1, 2, \dots, N$,
\begin{equation*}
\begin{aligned}
\begin{cases}
    	\vspace{4pt} \displaystyle \partial_t v_{i} + \inf_{a_{it} \in \mathbb R} \left( a_{it} \partial_{x_i} v_i + \frac{1}{2} a_{it}^2 \right)  +  \sum_{j \ne i}^{N} a_{jt} \partial_{x_j} v_{i} + \Delta v_{i} + \frac{k}{N} \sum_{j \ne i}^{N} \left(\left(e_i-e_j\right)^\top x^{(N)} \right)^2= 0, \\
	\displaystyle v_{i} \left(T, x^{(N)} \right) = 0.
\end{cases}
\end{aligned}
\end{equation*}
Note that with $a_{it} = - \partial_{x_i} v_i \left(t, x^{(N)} \right)$ for each $i = 1, 2, \dots, N$, the term in the infimum attains the optimal value and thus the HJB equation can be reduced to 
\begin{align}
\begin{cases}
	\vspace{4pt} \displaystyle \partial_t v_{i} - \frac{1}{2} \left(\partial_{x_i} v_i \right)^2  - \sum_{j \ne i}^{N} \partial_{x_j} v_j \partial_{x_j} v_{i} + \Delta v_{i} + \frac{k}{N} \sum_{j \ne i}^{N} \left(\left(e_i-e_j\right)^\top x^{(N)} \right)^2= 0,\\
	\displaystyle v_{i} \left(T, x^{(N)} \right) = 0.
\end{cases}
\label{V}
\end{align}
Then, the value functions $V$ of $N$-player game defined by \eqref{eq:value_i} is $V_i(x^{(N)}) = v_{i}(0, x^{(N)})$ for all $i=1, 2, \dots, N$. Moreover, the path and the control under the equilibrium are given by
$$d \hat X^{(N)}_{it} = - \partial_{x_i} v_{i} \left(t, \hat X^{(N)}_t \right) dt + dW^{(N)}_{it} + d \tilde{W}_t,$$
and
$$ \hat{\alpha}^{(N)}_{it} = - \partial_{x_i} v_{i} \left(t, \hat{X}^{(N)}_t \right)$$
for $i = 1, 2, \dots, N$.
The proof is the application of It\^o's formula and the details are omitted here.
Due to its LQG structure, the value function leads to a quadratic function of the form
$$v_{i} \left(t, x^{(N)} \right) = \left(x^{(N)} \right)^\top A_{i}(t) x^{(N)} + \left(x^{(N)} \right)^\top B_{i}(t) + C_{i}(t).$$ 
Plugging $V_{i}$ into \eqref{V}, and matching the coefficient of variables, we get the Riccati system of ODEs in \eqref{eq:ABC} and the desired results are obtained.
\end{proof}

\subsection{Proof of Proposition \ref{p:ABCexist}: Reduced Riccati form for the equilibrium}
\label{s:simplify}
At present, the MFG and the corresponding $N$-player game can be characterized by Proposition \ref{p:mfg} and Lemma \ref{l:riccati-N}, respectively. One of our primary objectives is to examine the convergence of the representative optimal path $\hat X_{1t}^{(N)}$ generated by the $N$-player game defined in \eqref{eq:ABC}-\eqref{eq:Xihat} to the optimal path $\hat X_t$ of the MFG described in Proposition \ref{p:mfg}.

It should be noted that $\hat X_t$ is solely dependent on the function $a(t)$, as indicated in the ODE \eqref{eq:a_odep}. In contrast, $\hat X_{1t}^{(N)}$ depends on $O(N^3)$ many functions derived from the solutions of a substantial Riccati system \eqref{eq:ABC} involving matrices $(A_{it}, B_{it}: i = 1, 2, \dots, N)$. Consequently, comparing these two processes meaningfully becomes an exceedingly challenging task without gaining further insight into the intricate structure of the Riccati system \eqref{eq:ABC}.
\begin{proof}
[{\bf Proof of Proposition \ref{p:ABCexist}}]
Inspired from the setup in \cite{JLSY22} and \cite{HY21}, we may seek a pattern for the matrix $A_i$ in the following form:
\begin{equation}
\label{eq:magic}
    (A_{i})_{pq} = \begin{cases}
	a_{1}(t), & \text{ if } p=q=i,\\
	a_{2}(t), & \text{ if } p=q\ne i,\\
	a_{3}(t), & \text{ if } p\ne q, p = i \text{ or } q = i, \\
	a_{4}(t), & \text{ otherwise}.
\end{cases}
\end{equation}
The next result justifies the above pattern: the $N^2$ entries of the matrix $A_{i}$ can be embedded to a $2$-dimensional vector space no matter how big $N$ is.

For the Riccati system \eqref{eq:ABC}, with the given of $A_i$ and suppose each function in $A_i$ is continuous on $[0, T]$, it is obvious to see that $B_{i} = 0$ for all $t \in [0, T]$ and for all $i = 1,2, \dots, N$. Note that in this case, for $i = 1, 2, \dots, N$, the  optimal control is given by
\begin{equation*}
    \hat{\alpha}_i  = - 2\sum_{j=1}^N (A_{i})_{ij} \hat X_{jt}^{(N)} = -2 \left(A_{i}\right)_i^\top \hat X^{(N)}_t,
\end{equation*}
where $(A)_i$ is the $i$-th column of matrix $A$.
		
Plugging the pattern \eqref{eq:magic} into the differential equation of $A_{i}$, we obtain the following system of ODEs:
\begin{equation*}
\begin{aligned}
\begin{cases}
    \vspace{4pt} \displaystyle a_{1}' -2a_{1}^2 - 4(N-1)a_{3}^2 + \frac{N-1}{N} k = 0, \\
	\vspace{4pt} \displaystyle a_{2}' - 2a_{3}^2 - 4a_{1}a_{2} - 4(N-2)a_{3}a_{4} + \frac{k}{N} = 0, \\
	\vspace{4pt} \displaystyle a_{3}' - 2a_{1}a_{3} - 4a_{1}a_{3}- 4(N-2)a_{3}^2 - \frac{k}{N} = 0, \\
	\vspace{4pt} \displaystyle a_{3}' -2a_{1}a_{3} - 4a_{2}a_{3} - 4(N-2)a_{3}a_{4} - \frac{k}{N} = 0, \\
	\displaystyle a_{4}' -2a_{3}^2 -4a_{2}a_{3} - 4a_{1}a_{4} - 4(N-3)a_{3}a_{4} = 0   
\end{cases}
\end{aligned}
\end{equation*}
with the terminal conditions
$$a_1(T) = a_2(T) = a_3(T) = a_4(T) = 0.$$
It is worth noting that there are two ODEs for $a_3$, and the two expressions should be equal, thus
\begin{equation*}
    a_{1} a_{3} + (N-2) a_{3}^2 = a_{2} a_{3} + (N-2) a_{3} a_{4},
\end{equation*}
which implies that $\left(a_{1} +(N-2)a_{3}\right)' =\left( a_{2} + (N-2) a_{4}\right)'$ or
\begin{equation*}
\begin{aligned}
	& 2a_{1}^2+2(N-2)a_{1}a_{3} + 4(N-1)a_{3}^2 +4(N-2)a_{2}a_{3} +4(N-2)^2a_{3}a_{4} -\frac{k}{N} \\
	= \ & 2(N-1)a_{3}^2 +4a_{1}a_{2}+4(N-2)(a_{2}a_{3}+a_{3}a_{4}+a_{1}a_{4}) +4(N-2)(N-3) a_{3}a_{4} -\frac{k}{N}.
\end{aligned}
\end{equation*}
After combining terms and substituting $a_{2}+(N-2)a_{4}$ with $a_{1} +(N-2) a_{3}$, we get 
\begin{equation*}
    a_{1}^2 +(N-2)a_{1}a_{3} - (N-1)a_{3}^2 = 0,
\end{equation*}
which yields $a_{3} = a_{1}$ or $a_{3} = -\frac{1}{N-1} a_{1}$. Note that, since $a_1$ and $a_3$ satisfies different differential equations, it follows that $a_{3}\ne a_{1}$. Hence, we can conclude that $a_{3} = -\frac{1}{N-1} a_{1}$. Next, from the equation $a_1 + (N-2) a_3 = a_2 + (N-2)a_4$, we have
\begin{equation*}
    a_4 = \frac{1}{N-2} a_1 + a_3 - \frac{1}{N-2} a_2. 
\end{equation*}
In conclusion, for $i = 1, 2, \dots, N$, $A_{i}$ has the following expressions:
\begin{equation*}
	(A_{i})_{pq} =
\begin{cases}
	\vspace{4pt} \displaystyle a_{1}(t), & \text{ if } p=q=i,\\
	\vspace{4pt} \displaystyle a_{2}(t), & \text{ if } p=q\ne i,\\
	\vspace{4pt} \displaystyle -\frac{1}{N-1}a_{1}(t), & \text{ if } p\ne q, p = i \text{ or } q = i,\\
	\displaystyle \frac{1}{(N-1)(N-2)}a_{1}(t)- \frac{1}{N-2}a_{2} (t), & \text{ otherwise},
\end{cases}
\end{equation*}
where $a_1$ and $a_2$ satisfies the system of ODEs
\eqref{eq:a_12}
\begin{equation}
\label{eq:a_12}
\begin{cases}
	\vspace{4pt} \displaystyle a_{1}' -\frac{2(N+1)}{N-1}a_{1}^2 + \frac{N-1}{N} k = 0, \\
	\vspace{4pt} \displaystyle a_{2}' +\frac{2}{(N-1)^2}a_{1}^2 -\frac{4N}{N-1}a_{1}a_{2} + \frac{k}{N} = 0,\\
	\displaystyle a_{1}(T) = a_{2}(T) = 0.
\end{cases}
\end{equation}

The existence and uniqueness of $A_i$ in \eqref{eq:ABC} are equivalent to the existence and uniqueness of \eqref{eq:a_12}. Firstly, the existence, uniqueness, and boundness of $a_1$ in \eqref{eq:a_12} is from the same argument for $a$ in \eqref{eq:ode1}, which is shown as the proof of Lemma \ref{l:solution of Riccati system} in Appendix. The explicit solution of $a_1$ is given by
\begin{equation*}
    a_1 (t) = \sqrt{\frac{k}{2} \frac{(N-1)^2}{N(N+1)}} \frac{1 - e^{-2 \sqrt{2} \sqrt{\frac{N+1}{N} k}(T-t)}}{1 + e^{-2 \sqrt{2} \sqrt{\frac{N+1}{N} k}(T-t)}}
\end{equation*}
for all $t \in [0, T]$. Next, with the given of $a_1$, the existence, uniqueness, and boundness of $a_2$ in \eqref{eq:a_12} is guaranteed by Theorem 12.1 in \cite{antsaklis2006}.
Therefore, we can express the equilibrium paths and associated controls as the following: 
\begin{equation}
\label{eq:XihatN}
    d \hat X_{it}^{(N)}  = - 2 a_{1}^N(t)  \left(\hat X_{it}^{(N)}  -\frac{1}{N-1}\sum_{j \ne i}^N  \hat X_{jt}^{(N)} 
	\right) dt + dW_{it}^{(N)} + d \tilde{W}_t,
\end{equation}
and
$$ \hat{\alpha}_{it}^{(N)} = - 2a_{1}^N (t) \left(\hat X_{it}^{(N)} -\frac{1}{N-1}\sum_{j \ne i}^N  \hat X_{jt}^{(N)} \right)$$
respectively for $i = 1, 2, \dots, N$, where $a_1^N$ is the solution to the ODE for $a_1$ in \eqref{eq:a_12}.
This concludes Proposition \ref{p:ABCexist}.
\end{proof}

\section{Further remark}
We have now established Proposition \ref{p:mfg} concerning the MFG in Section \ref{s:section3} and Proposition \ref{p:ABCexist} regarding the $N$-player game in Section \ref{s:section4}. With these propositions proven, we are now able to conclude the proof of Theorem \ref{t:main1}, which was presented in Section \ref{s:pthm}.

\section{Appendix}
\label{s:appendix}

\begin{lemma}
\label{l:was1}
Let $\mathbb W_p$ be the $p$-Wasserstein metric. If $X$ and $Y$ are two real-valued random variables and $c$ is a constant, then  
\begin{equation}
\label{eq:lwas1_1}
\mathbb W_p(\mathcal L(X), \mathcal L(Y)) =\mathbb W_p(\mathcal L(X+c), \mathcal L(Y+c)).
\end{equation}
Moreover, if $\alpha = \{\alpha_i: i \in \mathbb N\}$ is a sequence of random variables, then
\begin{equation}
\label{eq:lwas1_2}
\mathbb W_p \left(
\frac 1 N \sum_{i=1}^N \delta_{\alpha_i  + c}, \mathcal L(Y+c)
\right)
= 
\mathbb W_p \left(
\frac 1 N \sum_{i=1}^N \delta_{\alpha_i }, \mathcal L(Y)
\right)
.\end{equation}
\end{lemma}

\begin{proof}
By definition of the $p$-Wasserstein metric, we have:
$$\mathbb W_p(\mathcal L(X), \mathcal L(Y)) = \left(\inf_{\pi \in \Pi(\mathcal L(X), \mathcal L(Y))} \int_{\mathbb{R}^2} |x - y|^p d\pi(x,y)\right)^{\frac{1}{p}},$$
where $\Pi(\mathcal L(X), \mathcal L(Y))$ is the set of all joint probability measures with marginals $\mathcal L(X)$ and $\mathcal L(Y)$.
Similarly,
$$\mathbb W_p(\mathcal L(X+c), \mathcal L(Y+c)) = \left(\inf_{\pi \in \Pi(\mathcal L(X+c), \mathcal L(Y+c))} \int_{\mathbb{R}^2} |x - y|^p d\pi(x,y)\right)^{\frac{1}{p}},$$
where $\Pi(\mathcal L(X+c), \mathcal L(Y+c))$ is the set of all joint probability measures with marginals 
$\mathcal L(X+c)$ and $\mathcal L(Y+c)$.

Now, consider the mapping $\Phi:\mathbb{R}^2 \to \mathbb{R}^2$ given by $\Phi(x, y) = (x + c, y+c)$. 
For any $\pi \in \Pi(\mathcal L(X), \mathcal L(Y))$, the pushforward measure of $\pi$ under $\Phi$ belongs to 
$\Pi(\mathcal L(X+c), \mathcal L(Y+c))$, i.e., $\pi' = \Phi_{*} \pi \in \Pi(\mathcal L(X+c), \mathcal L(Y+c))$.
Thus, we have 
$$\Phi_{*}  \Pi(\mathcal L(X), \mathcal L(Y)) \subset  \Pi(\mathcal L(X+c), \mathcal L(Y+c)).$$
Moreover, $\Phi$ is bijective and measure preserving, then
$$\int_{\mathbb{R}^2} |x - y|^p d\pi'(x,y) = \int_{\mathbb{R}^2} |(x+c) - (y+c)|^p d\pi(x,y)= \int_{\mathbb{R}^2} |x - y|^p d\pi(x,y).$$
Therefore, we know that
\begin{equation*}
\begin{aligned}
\mathbb W_p^p \left( \mathcal L(X), \mathcal L(Y) \right) & =
\inf_{\pi \in \Pi(\mathcal L(X), \mathcal L(Y))} \int_{\mathbb{R}^2} |x - y|^p d\pi(x,y) \\ 
& = \inf_{\pi \in \Pi(\mathcal L(X), \mathcal L(Y))} \int_{\mathbb{R}^2} |x - y|^p d \Phi_{*}\pi(x,y) \\
& = \inf_{\pi' \in \Phi_{*} \Pi(\mathcal L(X), \mathcal L(Y))} \int_{\mathbb{R}^2} |x - y|^p d \pi'(x,y) \\ 
& \geq \mathbb W_p^p (\mathcal L(X+c), \mathcal L(Y+c)).
\end{aligned}
\end{equation*}
by the definition of the  $p$-Wasserstein metric. If we apply the above inequality to $X' = X+c$, $Y' = Y+c$, and $c' = -c$, the opposite inequality is provided. Thus, it completes the proof of \eqref{eq:lwas1_1}.

Next, we note that
$$\frac 1 N \sum_{i=1}^N \delta_{\alpha_i  + c} = \mathcal L(\alpha_u +c| \alpha),$$
where $u$ be a uniform random variable on $\{1, 2, \dots, N\}$ independent to $\alpha$. Using \eqref{eq:lwas1_1}, we conclude \eqref{eq:lwas1_2} from
\begin{equation*}
\begin{aligned}
\mathbb W_p \left(
\frac 1 N \sum_{i=1}^N \delta_{\alpha_i  + c}, \mathcal L(Y+c)
\right)
& = 
\mathbb W_p \left(
\mathcal L(\alpha_u +c| \alpha), \mathcal L(Y+c)
\right)
\\
& = \mathbb W_p \left(
\mathcal L(\alpha_u | \alpha), \mathcal L(Y)
\right)
\\
&= 
\mathbb W_p \left(
\frac{1}{N} \sum_{i=1}^N \delta_{\alpha_i }, \mathcal L(Y)
\right)
.
\end{aligned}
\end{equation*}
\end{proof}

\begin{lemma}
\label{l:solution of Riccati system}
Under the Assumption \ref{a:asm2}, there exists a unique solution $\left(a(t), b(t), c(t), d(t) : t \in [0, T] \right)$ for the Riccati system of ODEs \eqref{eq:ode1}-\eqref{eq:terminal conditions} and the solution can given explicitly by
\begin{equation*}
\begin{cases}
\vspace{4pt}
\displaystyle a(t) = \sqrt{\frac{k}{2}} \frac{1 - e^{-2 \sqrt{2k} (T-t)}}{1 + e^{-2 \sqrt{2k} (T-t)}}, \\ 
\vspace{4pt}
\displaystyle b(t) = \int_t^T \left(4a(s) c(s) - 2a^2(s) \right) ds, \\ 
\vspace{4pt}
\displaystyle c(t) = k \int_t^T e^{\int_t^s -4 a(r) dr} ds, \\
\displaystyle d(t) = \int_t^T \left(b(s) + 2c(s) \right) ds. 
\end{cases}
\end{equation*}
\end{lemma}

\begin{proof}
    Firstly, with the given of $k > 0$, we can solve the ODE
    $$a'(t) - 2a^2(t) + k = 0, \quad a(T) = 0$$
    explicitly by the method of separating variables. Note that with the differential form, we have
    $$\frac{da}{\left(\sqrt{2}a - \sqrt{k} \right)\left(\sqrt{2}a + \sqrt{k} \right)} = \frac{1}{2 \sqrt{k}} \left(\frac{1}{\sqrt{2}a - \sqrt{k}} - \frac{1}{\sqrt{2}a + \sqrt{k}}  \right) da = dt.$$
    It follows that
    $$\ln \left( \left\vert \frac{\sqrt{2}a - \sqrt{k}}{\sqrt{2}a + \sqrt{k}} \right\vert \right) = 2 \sqrt{2k} t + C_1$$
    for some constant $C_1$ by taking integration on both sides. Thus by calculation, we obtain
    $$a(t) = \sqrt{\frac{k}{2}} \frac{1 - C_2e^{2 \sqrt{2k} t}}{1 + C_2e^{2 \sqrt{2k} t}}$$
    for some constant $C_2$ to be determined.
    Since $a(T) = 0$, it yields that $C_2 = e^{-2\sqrt{2k} T}$ and thus
    $$a(t) = \sqrt{\frac{k}{2}} \frac{1 - e^{-2 \sqrt{2k} (T-t)}}{1 + e^{-2 \sqrt{2k} (T-t)}}.$$
    It is easily to verify that $a(\cdot)$ is in $C^{\infty}([0, T])$ and is bounded. With the given of $a$, the functions $\left(b, c, d \right)$ in the Riccati system \eqref{eq:ode1}-\eqref{eq:terminal conditions} is a coupled linear system, and thus their existence, uniqueness, and boundedness are given by Theorem 12.1 in \cite{antsaklis2006}.
\end{proof}

\bibliographystyle{plain}
\bibliography{main}

\end{document}